\documentclass[11pt]{article}
\usepackage{graphicx}
\usepackage{amsfonts}
\usepackage{amsmath}
\usepackage{amsthm}
\usepackage{amssymb}
\usepackage{fancyhdr}
\usepackage{indentfirst}
\usepackage{titlesec}
\usepackage{hyperref}
\usepackage{abstract}
\usepackage{abstract}

\newcommand{\R}{\mathbb{R}}
\newcommand{\mH}{\mathbb{H}}
\newcommand{\mL}{\mathcal{L}}
\newcommand{\CH}{\mathbb{H}_{\mathbb{C}}}
\newcommand{\vp}{\varphi}
\newcommand{\I}{\mathcal{I}}
\newcommand{\dU}{\Delta U}
\newcommand{\dv}{\Delta v}
\newcommand{\dch}{\Delta\chi}
\newcommand{\dps}{\Delta\psi}
\newcommand{\al}{\alpha}
\newcommand{\be}{\beta}
\newcommand{\de}{\delta}
\newcommand{\la}{\lambda}
\newcommand{\pr}{\prime}
\newcommand{\Ga}{\Gamma}
\newcommand{\ga}{\gamma}
\newcommand{\mA}{\mathcal{A}}
\newcommand{\ep}{\epsilon}
\newcommand{\mC}{\mathcal{C}}
\newcommand{\mW}{\mathcal{W}}
\newcommand{\Om}{\Omega}
\newcommand{\lan}{\langle}
\newcommand{\ran}{\rangle}
\newcommand{\M}{\mathcal{M}}
\newcommand{\om}{\omega}
\newcommand{\ti}{\tilde}
\newcommand{\si}{\sigma}

\newtheorem{theorem}{Theorem}[section]
\newtheorem{proposition}[theorem]{Propostion}

\newtheorem{remark}[theorem]{Remark}

\theoremstyle{definition}

\theoremstyle{plain}
\newtheorem{lemma}[theorem]{Lemma}

\numberwithin{equation}{section}

\begin{document}

%
%

\pagestyle{headings}
\renewcommand{\headrulewidth}{0.4pt}

\title{\textbf{Convexity of reduced energy and mass angular momentum inequalities}}
\author{Richard Schoen and Xin Zhou\footnote{The first author was partially supported by NSF grant DMS-1105323}\\}
\date{}
\maketitle


\begin{abstract}
In this paper, we extend the work in \cite{D}\cite{ChrusLiWe}\cite{ChrusCo}\cite{Co}. We weaken 
the asymptotic conditions on the second fundamental form, and we also give an $L^{6}-$norm bound for the difference between general data and Extreme Kerr data or Extreme Kerr-Newman data by proving convexity of the renormalized Dirichlet energy when the target has non-positive curvature. In particular, we give the first proof of the strict mass/angular momentum/charge
inequality for axisymmetric Einstein/Maxwell data which is not identical with the extreme Kerr-Newman solution. 
\end{abstract}



\section{Introduction}\label{introduction}
An interesting question about solutions of the Einstein equations is whether the angular
momentum (and charge for the Einstein/Maxwell case) can be bounded by the mass for 
physically reasonable solutions. This is
true for the Kerr and Kerr-Newman black hole solutions which are stationary. For dynamical, axisymmetric
solutions some general results have been obtained, first by S. Dain \cite{D} and later by
other authors \cite{ChrusLiWe}\cite{ChrusCo}\cite{Co} over the past several years. In this 
paper we introduce a new method for obtaining such inequalities which is technically
simpler and which provides sharper results in many cases. We apply this method to both the vacuum black hole case and to the Einstein/Maxwell black hole case. An interesting feature of
our method is that it provides a quantitative lower bound on the gap in the inequality
in terms of an $L^6$ measure of the distance between the dynamical solution and the
comparison stationary solution. As such it readily handles the borderline case, and
provides an extremal characterization of the Kerr and Kerr-Newman solutions. In this paper
we deal with the reduction of the initial data to a mapping and we state our theorems in
terms of the mapping. For the corresponding statements in terms of physical quantities we
refer to Theorem 1.1 of \cite{ChrusLiWe} for the vacuum case and to Theorem 1.1 of \cite{Co}
for the Einstein/Maxwell case.

It is well known that the Dirichlet energy for mappings from compact manifolds into negatively curved Riemannian manifolds has a strong convexity property along geodesic deformations \cite{S}. Here we will prove a similar convexity result for the normalized Dirichlet energy of certain singular mappings to negatively curved Riemannian manifold arising from mathematical general relativity (see \cite{D}\cite{We}\cite{ChrusLiWe}). We will use this convexity to show that singular harmonic maps are unique in a class of maps with finite reduced energy and the same asymptotic singular behavior. Moreover, we can control the $L^{6}$ norm of the distance between any such map and the singular harmonic map by the reduced energy gap.

On $\mathbb{R}^{3}$, we use $(\rho, \varphi, z)$ to denote cylindrical coordinates, and $(r, \theta, \phi)$ to denote spherical coordinates. We use $\Gamma$ to denote the $z-$axis which is given by $\{\rho\equiv 0\}$. We define $g$ by
\begin{equation}
g=2\log\rho,
\end{equation}
and note that $g$ is the potential of a uniform charge distribution on $\Gamma$. In particular $g$ is  harmonic on $\mathbb{R}^{3}\setminus\Gamma$. Now we are interested in the mapping $(X, Y): \Om\subset\mathbb{R}^{3}\rightarrow\mathbb{H}^{2}$, where $\mathbb{H}^{2}=\{(X, Y)\in\mathbb{R}^{2}, X>0\}$ is the hyperbolic right half plane with metric $ds^{2}_{-1}=\frac{dX^{2}+dY^{2}}{X^{2}}$. Since $X>0$, we can rewrite $X$ as $X=e^{g+x}$, or equivalently
\begin{equation}\label{x and X}
x=\log X-g.
\end{equation}
We are interested in the following functional discussed in \cite{D}.
\begin{equation}\label{reduced energy}
\mathcal{M}_{\Om}(x, Y)=\int_{\Om}|\partial x|^{2}+e^{-2g-2x}|\partial Y|^{2}d\mu.
\end{equation}
We denote $\M(x, Y)=\M_{\mathbb{R}^{3}}(x, Y)$. The motivation to study this functional is that the \emph{extreme Kerr Solution} of the Einstein vacuum equations gives rise to a local critical point of the above functional. The extreme Kerr solution corresponds to the map $(X_{0}, Y_{0})$, or equivalently $(x_{0}, Y_{0})$ where $x_{0}=\log X_{0}-g$, which in spherical coordinates, is given by (see \cite{D})
\begin{equation}\label{Extreme Kerr}
X_{0}=\big(\tilde{r}^{2}+|J|+\frac{2|J|^{3/2}\tilde{r}\sin^{2}\theta}{\Sigma}\big)\sin^{2}\theta,\ \ Y_{0}=2J(\cos^{3}\theta-3\cos\theta)-\frac{2J^{2}\cos\theta\sin^{4}\theta}{\Sigma},
\end{equation}
and
\begin{equation}
\tilde{r}=r+\sqrt{|J|},\ \ \Sigma=\tilde{r}^{2}+|J|\cos^{2}\theta,
\end{equation}
where the number $J$ corresponds to the angular momentum of the spacetime corresponding
to $(X_{0}, Y_{0})$.

Now we are interested in the class of $(x, Y)$ such that functional $\mathcal{M}$ in equation (\ref{reduced energy}) is well-defined, finite and physically corresponds to an axisymmetric initial data set for the vacuum Einstein equations\footnote{We refer this physical background to \cite{D} and \cite{D2}.}. In fact, we are interested in a class of data which can be written as variations of the Kerr solutions. Denote
\begin{equation}
x=x_{0}+\alpha,\ \ \ Y=Y_{0}+y.
\end{equation}
Let $\alpha\in H^{1}(\mathbb{R}^{3})$, which is the completion of $C^{\infty}_{c}(\mathbb{R}^{3}\setminus\{0\})$ under the norm
\begin{equation}\label{norm for x}
\|\alpha\|_{1}=\big(\int_{\mathbb{R}^{3}}|\partial\alpha|^{2}d\mu\big)^{1/2},
\end{equation}
and $y\in H^{1}_{0, X_{0}}(\mathbb{R}^{3}\setminus\Gamma)$, which is the completion of $C^{\infty}_{c}(\mathbb{R}^{3}\setminus\Gamma)$ under the norm
\begin{equation}\label{norm for y}
\|y\|_{1, X_{0}}=\big(\int_{\mathbb{R}^{3}}X_{0}^{-2}|\partial y|^{2}d\mu\big)^{1/2}.
\end{equation}
Here $d\mu$ denotes the Euclidean volume measure.

We will give a simplified proof of a strengthening of Theorem 1.2 of \cite{D}.
\begin{theorem}\label{main theorem}
The functional $\mathcal{M}(x, Y)$ achieves a global minimum at the Extreme Kerr solution 
$(x_{0}, Y_{0})$ over all $\{x=x_{0}+\alpha, Y=Y_{0}+y\}$, where 
$\alpha\in H^{1}(\mathbb{R}^{3})$ with $\alpha_{-}=\inf\{0, \alpha\}\in L^{\infty}(\mathbb{R}^{3})$, and $y\in H^{1}_{0, X_{0}}(\mathbb{R}^{3}\setminus\Gamma)$, that is, for any such $(x,Y)$
\begin{equation}\label{main inequality}
\mathcal{M}(x, Y)\geq\mathcal{M}(x_{0}, Y_{0}).
\end{equation}
Furthermore, we have the following gap bound,
\begin{equation}\label{gap inequality}
\mathcal{M}(x, Y)-\mathcal{M}(x_{0}, Y_{0})\geq C\{\int_{\R^{3}}d_{-1}^{6}\big((X, Y), (X_{0}, Y_{0})d\mu\big)\}^{1/3}
\end{equation}
where $d_{-1}(\cdot,\cdot)$ is the distance function on $\mathbb H^2$.
\end{theorem}

\begin{remark}
Here the condition $\alpha_{-}\in L^{\infty}$ is needed to insure that $\mathcal{M}(x, Y)$ to be finite for $y\in H^{1}_{0, X_{0}}$. We do not need the $L^{\infty}$ condition for $X_{0}^{-1}y$ which is assumed in Theorem 1.2 of \cite{D}, since we do not need to construct a minimizer of 
$\mathcal{M}$ in our proof. 
\end{remark}

In \cite{Chrus1}, P. Chru\'sciel generalized the class of axially symmetric initial data which admit a representation as a mapping to $\mathbb H^2$ and extended a theorem of D. Brill \cite{Brill} to prove the positive mass theorem for data in this class. The mass/angular momentum inequality for this class was obtained by P. Chru\'sciel, Y. Y. Li, and G. Weinstein \cite{ChrusLiWe}. In Section \ref{chrusciel} we extend our method to recover their theorem in a stronger form including
the gap estimate. This is done in Theorem \ref{main theorem2}. In addition to obtaining the
$L^6$ lower bound for the gap, we weaken the asymptotic assumptions, requiring the second
fundamental form $h$ to decay strictly faster than $r^{-3/2}$ while the results of \cite{ChrusLiWe}
require decay strictly faster than $r^{-5/2}$.

In Section \ref{e-m case} we apply our method to the case of Einstein/Maxwell black hole
data. In this case the target manifold for the associated mapping is the complex hyperbolic
plane $\mathbb H^2_{\mathbb C}$ (four real dimensions). In Theorem \ref{main theorem3}
we give an extension of Theorem \ref{main theorem} to bound the gap in the reduced energy between a general map to $\mathbb H^2_{\mathbb C}$ in an appropriate asymptotic class (see (\ref{conditions of dPsi in E/M case})) and the harmonic map corresponding to the extremal
Kerr-Neuman solution. In order to prove mass/angular momentum inequalities for black hole Einstein/Maxwell initial data, we extend our method in Section \ref{cc data} to cover a class of
initial data introduced by Chru\'sciel and J. Costa \cite{ChrusCo}, \cite{Co}. This requires
a careful examination of the asymptotic conditions which is given in \ref{asymptotic}.
The main theorem extending the results of \cite{ChrusCo} and \cite{Co} is Theorem \ref{main theorem4}. Our theorem includes a lower bound on the gap and therefore also implies the borderline case which gives a characterization of the Kerr-Newman solution. This does not
appear to follow from \cite{ChrusCo} and \cite{Co}.


\section{Convexity for $\mathcal{M}$}\label{convexity}

The motivation to study convexity properties of $\mathcal{M}$ comes from the relation between $\mathcal{M}$ and the Dirichlet energy $E$, which is defined for $(X, Y): \mathbb{R}^{3}\rightarrow\mathbb{H}^{2}$ by
\begin{equation}\label{harmonic energy}
E(X, Y)=\int_{\mathbb{R}^{3}}\frac{|\partial X|^{2}+|\partial Y|^{2}}{X^{2}}d\mu.
\end{equation}
Here $E$ is just the standard harmonic map energy\footnote{See definition and properties in \cite{S}} for mapping $(X, Y): \mathbb{R}^{3}\rightarrow\mathbb{H}^{2}$.

\subsection{Convexity of the Dirichlet energy}

Now let us first discuss a general result. Let $(M, g)$ be a general $n$ dimensional Riemannian manifold, and $\Omega\subset(M, g)$ an open subset with or without boundary. Let $(N, h)$ be a target Riemannian manifold, and $u_{0},\ u_{1}:\Omega\rightarrow(N, h)$ be $C^{2}$ mappings. Now connect them by a $C^{2}$ family of mappings $F: \Omega\times[0, 1]\rightarrow(N, h)$. We denote the energy restricted to maps on $\Omega$ by $E_{\Omega}$. We let $F_t$ denote the map with $t$ fixed, and we consider the second variation of the energy\footnote{The results went back to Section 3 of \cite{S}.} of $F_t$. Denote the variational vector field by 
$V=F_{*}(\frac{\partial}{\partial t})$, then we have the \textbf{second variation formula}:
\begin{equation}\label{second variation of E}
\begin{split}
\frac{d^{2}}{dt^{2}}E_{\Omega}(F_{t}) & =2\int_{\Omega}\big[\sum_{\alpha=1}^{n}\|\nabla^{N}_{(F_{t})_{*}(e_{\alpha})}V\|^{2}_{h}\\
                                      &-\sum_{\alpha=1}^{n}R^{N}(V, (F_{t})_{*}(e_{\alpha}), V, (F_{t})_{*}(e_{\alpha}))-div_{F(M)}(\nabla^{N}_{V}V)\big]dvol_{M},
\end{split}
\end{equation}
where $\{e_{\alpha}\}_{\alpha=1}^{n}$ is a local orthonormal basis on $(\Omega, g)$. So if the target manifold $(N, h)$ has non-positive sectional curvature, then the second term in the above integral is non-negative. If we can choose $F_{t}$ to be a geodesic deformation, i.e $F_{t}(x): [0, 1]\rightarrow(N, h)$ is a geodesic for any fixed $x\in\Omega$, then we know that $\nabla^{N}_{V}V\equiv 0$, so the last term in the above integral is zero. So we get $\frac{d^{2}}{dt^{2}}E_{\Omega}(F_{t})\geq 0$, which is the convexity for the Dirichlet energy under geodesic deformations.

Moreover, we have a refined estimate. In the second variation formula (\ref{second variation of E}), the third term in the integrand is zero, and the second term is nonnegative. To deal with the first term, we will use the following Kato inequality,
\begin{lemma}
 If $e$ and $V$ are two tangent vector fields on $(N, h)$, then
\begin{equation}
 \|\nabla_{e}V\|_{h}\geq |\nabla_{e}\|V\|_{h}|.
\end{equation}
\end{lemma}
\begin{proof} We have
$$\nabla_{e}\|V\|_{h}=\frac{\lan\nabla_{e}V, V\ran_{h}}{\|V\|_{h}},$$
so by the Cauchy-Schwartz inequality, we get the desired result.
\end{proof}
Applying the above result to the first term in equation (\ref{second variation of E}),
\begin{displaymath}
 \begin{split}
  \sum_{\alpha=1}^{n}\|\nabla^{N}_{(F_{t})_{*}e_{\alpha}}V\|_{h}^{2} & \geq\sum_{\alpha=1}^{n}|\nabla^{N}_{(F_{t})_{*}e_{\alpha}}\|V\|_{h}|^{2}\\
  & =\sum_{\alpha=1}^{n}|\nabla^{M}_{e_{\alpha}}(\|V\|_{h}\circ F_{t})|^{2}.
 \end{split}
\end{displaymath}
Since $F_{t}$ is chosen to be a geodesic deformation, we know that
$$\|V\|_{h}(F_{t}(x))=dist_{h}(F_{0}(x), F_{1}(x))=dist_{h}(u_{0}(x), u_{1}(x)),$$
where $dist_{h}$ is the distance function of $(N, h)$. Now putting this into equation (\ref{second variation of E}), we have the \textbf{refined second variation formula}:
\begin{equation}\label{refined second variation for E}
 \frac{d^{2}}{dt^{2}}E_{\Omega}(F_{t})\geq2\int_{\Omega}\|\nabla dist_{h}(u_{0}, u_{1})\|^{2}_{g}dvol_{M}.
\end{equation}
If $u_{0}$ is a harmonic map, by integrating the above inequality twice with respect to the variable $t$, we can get an estimate of the $L^{2}$ norm of the gradient of the distance function $dist_{h}(u_{0}, u_{1})$ by the energy gap.

\subsection{Singular case}\label{singular case}

Now we will apply the same idea to our functional $\mathcal{M}$ under geodesic deformations. The first observation concerns the relation between $\mathcal{M}$ and $E$. Consider a compact open domain $\Omega\subset\mathbb{R}^{3}\setminus\Gamma$ and put condition (\ref{x and X}) into equation (\ref{harmonic energy}). By an integration by parts argument based on the fact that $g$ is harmonic, we get\footnote{This is also given by equation (66) of \cite{D}.}
\begin{equation}\label{E and M}
E_{\Omega}(X, Y)=\mathcal{M}_{\Omega}(x, Y)+\int_{\partial\Omega}\frac{\partial g}{\partial n}(g+2x)d\sigma,
\end{equation}
where $\mathcal{M}_{\Omega}$ is the functional $\mathcal{M}$ restricted to domain $\Omega$, $n$ is the unit outer normal of $\partial\Omega$, and $d\sigma$ the area element of $\partial\Omega$. Since $E$ and $\mathcal{M}$ only differ by a boundary integral, they must have the same critical points and thus we call $\mathcal{M}$ the \emph{reduced energy}. In fact, $\mathcal{M}$ is a regularization of $E$ in this special case since we are removing the infinite term 
$\int|\partial g|^{2}$ from $E$.

Now we obtain a convexity result for $\mathcal{M}_{\Om}$. We first choose our compact domain $\Om$ as an annulus region $A_{R, \epsilon}=B_{R}\setminus B_{\epsilon}$, where $B_{R}$ denotes the Euclidean ball of radius $R$ in $\R^{3}$. Denote $\Omega_{R, \epsilon}=A_{R, \epsilon}\setminus\mathcal{C}_{\epsilon}$ where $\mathcal{C}_{\epsilon}=\{\rho\leq\epsilon\}$ 
is the cylinder centered on the $z$ axis $\Ga$ of radius $\epsilon$. The definition of 
$H^{1}(\mathbb{R}^{3})$ and $H^{1}_{0, X_{0}}(\mathbb{R}^{3}\setminus\Gamma)$ motivate us to first consider functions $\alpha\in C^{\infty}_{c}(A_{R, \epsilon})$ and $y\in C^{\infty}_{c}(\Omega_{R, \epsilon})$, with $X=e^{g+x_{0}+\alpha}$ and $Y=Y_{0}+y$. Now consider a geodesic deformation
$$F: A_{R, \epsilon}\times[0, 1]\rightarrow\mathbb{H}^{2},$$
with $F_{0}=(X_{0}, Y_{0})$ and $F_{1}=(X, Y)$.
Denote $F_{t}=(X_{t}, Y_{t})$, $x_{t}=\log X_{t}-g$, and $y_{t}=Y_{t}-Y_{0}$.

Now we make an important observation that reduces the computational difficulty substantially. Since $y\in C^{\infty}_{c}(\Omega_{R, \epsilon})$, we know that on a neighborhood of $\mathcal{C}_{\epsilon}\cap A_{R, \epsilon}$, $Y\equiv Y_{0}$, and $X=X_{0}e^{\alpha}$. By basic hyperbolic geometry, we know that the geodesic from $(X_{0}, Y_{0})$ to $(X=X_{0}e^{\alpha}, Y=Y_{0})$ is given by
\begin{equation}\label{geodesic near axis}
X_{t}=X_{0}e^{t\alpha},\ \ \ Y_{t}=Y_{0}.
\end{equation}
By using equation (\ref{x and X}), we have that on a neighborhood of $\mathcal{C}_{\epsilon}\cap A_{R, \epsilon}$,
\begin{equation}\label{geodesic for x}
x_{t}=x_{0}+t\alpha.
\end{equation}
Now let us compute the second variation of the reduced energy $\mathcal{M}_{A_{R, \epsilon}}$
$$\frac{d^{2}}{dt^{2}}\mathcal{M}_{A_{R,\epsilon}}(x_{t}, Y_{t})=\frac{d^{2}}{dt^{2}}\mathcal{M}_{\Omega_{R,\epsilon}}(x_{t}, Y_{t})+\frac{d^{2}}{dt^{2}}\mathcal{M}_{A_{R, \epsilon}\cap\mathcal{C}_{\epsilon}}(x_{t}, Y_{t}).$$
For the first term, we use equation (\ref{E and M})
\begin{equation}\label{first part of second variation of M}
\begin{split}
\frac{d^{2}}{dt^{2}}\mathcal{M}_{\Omega_{R,\epsilon}}(x_{t}, Y_{t}) &=\frac{d^{2}}{dt^{2}}E_{\Omega_{R,\epsilon}}(X_{t}, Y_{t})-\frac{d^{2}}{dt^{2}}\int_{\partial\Omega_{R,\epsilon}}\frac{\partial g}{\partial n}(g+2x_{t})d\sigma;\\
             & =\frac{d^{2}}{dt^{2}}E_{\Omega_{R,\epsilon}}(X_{t}, Y_{t})-2\frac{d^{2}}{dt^{2}}\int_{\partial\Omega_{R,\epsilon}\cap\mathcal{C}_{R, \epsilon}}\frac{\partial g}{\partial n}x_{t}d\sigma;\\
             & =\frac{d^{2}}{dt^{2}}E_{\Omega_{R,\epsilon}}(X_{t}, Y_{t})-2\frac{d^{2}}{dt^{2}}\int_{\partial\Omega_{R,\epsilon}\cap\mathcal{C}_{R, \epsilon}}\frac{\partial g}{\partial n}(x_{0}+t\alpha)d\sigma;\\
             & =\frac{d^{2}}{dt^{2}}E_{\Omega_{R,\epsilon}}(X_{t}, Y_{t})\\
             & \geq 2\int_{\Omega_{R, \epsilon}}|\nabla dist_{-1}\big((X, Y), (X_{0}, Y_{0})\big)|^{2}d\mu.
\end{split}
\end{equation}
Here $dist_{-1}$ is the distance function on the hyperbolic plane $\mathbb{H}_{-1}$. The second $``="$ is because that $x_{t}\equiv x_{0}$ near $\partial A_{R, \epsilon}\cap\Omega_{R, \epsilon}$ since $\alpha$ is compactly supported in $A_{R, \epsilon}$. The third $``="$ is given by equation (\ref{geodesic for x}). The last $``="$ is because the second term there is linear in $t$. The last inequality $``\geq"$ comes from the convexity of the harmonic energy (\ref{refined second variation for E}) along geodesic paths.

Now we deal with the second part by direct calculation
\begin{equation}\label{second part of second variation of M}
\begin{split}
\frac{d^{2}}{dt^{2}}\mathcal{M}_{A_{R, \epsilon}\cap\mathcal{C}_{\epsilon}}(x_{t}, Y_{t}) &= \frac{d^{2}}{dt^{2}}\int_{A_{R, \epsilon}\cap\mathcal{C}_{\epsilon}}|\nabla x_{t}|^{2}+e^{-2g-2x_{t}}|\nabla Y_{t}|^{2}d\mu\\
             & =\frac{d^{2}}{dt^{2}}\int_{A_{R, \epsilon}\cap\mathcal{C}_{\epsilon}}|\nabla(x_{0}+t\alpha)|^{2}+e^{-2g-2(x_{0}+t\alpha)}|\nabla Y_{0}|^{2}d\mu\\
             & =\int_{A_{R, \epsilon}\cap\mathcal{C}_{\epsilon}}2|\nabla\alpha|^{2}+4\alpha^{2}e^{-2g-2(x_{0}+t\alpha)}|\nabla Y_{0}|^{2}d\mu\\
             & \geq \int_{A_{R, \epsilon}\cap\mathcal{C}_{\epsilon}}2|\nabla\alpha|^{2}d\mu\\
             & = 2\int_{A_{R, \epsilon}\cap\mathcal{C}_{\epsilon}}|\nabla dist_{-1}\big((X, Y), (X_{0}, Y_{0})\big)|^{2}d\mu.
\end{split}
\end{equation}
The second $``="$ comes from equation (\ref{geodesic for x}) again. The last $"="$ follows from the equation (\ref{geodesic near axis}) on $A_{R, \epsilon}\cap\mathcal{C}_{\epsilon}$ and the fact that the distance $d_{-1}\big((X, Y), (X_{0}, Y_{0})\big)=\alpha$.
\begin{remark}
We can put $\frac{d^{2}}{dt^{2}}$ into the integral because that the integrands are all uniformly integrable.
\end{remark}
Now combining the above inequalities, we get the desired convexity under geodesic deformation,
\begin{lemma}\label{convexity1} With $(X_0,Y_0)$ and $(X,Y)$ as above we have
\begin{equation}\label{convexity inequality1}
 \frac{d^{2}}{dt^{2}}\mathcal{M}_{A_{R, \epsilon}}(x_{t}, Y_{t})\geq 2\int_{A_{R, \epsilon}}|\nabla d_{-1}\big((X, Y), (X_{0}, Y_{0})\big)|^{2}d\mu.
\end{equation}
\end{lemma}


\section{Proof of Theorem 1.1}\label{proof of main theorem1}
In this section we give the proof of Theorem \ref{main theorem}.
\begin{proof}
For $\alpha\in H^{1}(\mathbb{R}^{3})$, $\alpha_{-}=\inf\{0, \alpha\}\in L^{\infty}(\mathbb{R}^{3})$, and $y\in H^{1}_{0, X_{0}}(\mathbb{R}^{3}\setminus\Gamma)$, by the definition of $H^{1}(\mathbb{R}^{3})$ and $H^{1}_{0, X_{0}}(\mathbb{R}^{3}\setminus\Gamma)$, we can choose two sequences of mappings $\{\alpha_{n}\in C^{\infty}_{c}(\mathbb{R}^{3}\setminus\{0\})\}_{n=1}^{\infty}$ and $\{y_{n}\in C^{\infty}_{c}(\mathbb{R}^{3}\setminus\Gamma)\}_{n=1}^{\infty}$, such that\footnote{See equation (\ref{norm for x}) and (\ref{norm for y})},
\begin{equation}\label{approximation sequence}
\|\alpha-\alpha_{n}\|_{1}\rightarrow 0,\ \ \ \|y-y_{n}\|_{1, X_{0}}\rightarrow 0.
\end{equation}
It is easy to see that
\begin{equation}\label{approximation}
\mathcal{M}(x_{n}, Y_{n})\rightarrow\mathcal{M}(x, Y),
\end{equation}
where $x_{n}=x_{0}+\alpha_n$, $Y_{n}=Y_{0}+y_{n}$, and $(x, Y)$ is given in Theorem \ref{main theorem}. We can further assume that there exist two sequences of positive numbers $\{R_{n}\rightarrow\infty\}_{n=1}^{\infty}$ and $\{\epsilon_{n}\rightarrow 0\}_{n=1}^{\infty}$, such that $\alpha_{n}\in C^{\infty}_{c}(A_{R_{n}, \epsilon_{n}})$, and $y_{n}\in C^{\infty}_{c}(\Omega_{R_{n}, \epsilon_{n}})$.

Now we would like to use the argument in the proof of uniqueness of harmonic mappings when the ambient manifold is negatively curved\footnote{See Section 3 of \cite{S}}. For fixed $n$, we focus on the region $A_{R_{n}, \epsilon_{n}}$ and $\Omega_{R_{n}, \epsilon_{n}}$. We will discard the sub-index $n$ in the following argument. There is a geodesic deformation 
$F_{t}: A_{R, \epsilon}\rightarrow\mathbb{H}^{2}$ from $(X_{0}, Y_{0})$ to $(X=X_{0}e^{\alpha}, Y=Y_{0}+y)$. We know that $\mathcal{M}_{A_{R, \epsilon}}(F_{t})$ is a convex function from above. Since $(X_{0}, Y_{0})$ is harmonic on $\mathbb{R}^{3}\setminus\Gamma$, we will show that $(x_{0}, Y_{0})$ is critical point of the reduced functional $\mathcal{M}_{A_{R, \epsilon}}$. In fact, we have\footnote{See equations (70)(71) in \cite{D}.}:
\begin{equation}\label{equation for x0}
\triangle\log X_{0}=-\frac{|\partial Y_{0}|^{2}}{X_{0}^{2}},
\end{equation}
\begin{equation}\label{equation for y0}
\triangle Y_{0}=2\frac{\lan\partial Y_{0}, \partial X_{0}\ran}{X_{0}}.
\end{equation}

\begin{lemma} At $t=0$ we have
\begin{equation}
\frac{d}{dt}\Big{|}_{t=0}\mathcal{M}_{A_{R, \epsilon}}(F_{t})=0.
\end{equation}
\end{lemma}
\begin{proof} We compute
\begin{displaymath}
\begin{split}
\frac{d}{dt}\Big{|}_{t=0} & \mathcal{M}_{A_{R, \epsilon}}(x_{t}, Y_{t})=\frac{d}{dt}\Big{|}_{t=0}\int_{A_{R, \epsilon}}|\partial x_{t}|^{2}+e^{-2g-2x_{t}}|\partial Y_{t}|^{2}d\mu\\
         & =2\int_{A_{R, \epsilon}}\lan\partial x_{0}, \partial x^{\prime}_{0}\ran- x_{0}^{\prime}e^{-2g-2x_{0}}|\partial Y_{0}|^{2}+e^{-2g-2x_{0}}\lan\partial Y_{0}, \partial Y^{\prime}_{0}\ran d\mu.
\end{split}
\end{displaymath}
Here we put the $\frac{d}{dt}$ into the integral in the second $``="$ since the integrand is uniformly integrable.

Taking $\la\ll\ep$, we separate $A_{R, \ep}$ into two parts $A_{R, \ep}\setminus\mC_{\la}$ and $A_{R, \ep}\cap \mC_{\la}$. Using that $(X_{0}, Y_{0})$ satisfies the Euler-Lagrange equations(\ref{equation for x0})(\ref{equation for y0}) for $\M$ to do integration by parts on $A_{R, \ep}\setminus\mC_{\la}$ where all functions are regular, and noticing the fact that $Y^{\pr}_{0}\equiv 0$ near $\mC_{\la}$, we have
$$\frac{d}{dt}\Big{|}_{t=0}\mathcal{M}_{A_{R, \epsilon}}(x_{t}, Y_{t}) =2\int_{\{\rho=\la\}\cap A_{R, \ep}}\frac{\partial x_{0}}{\partial n}\cdot\al d\sigma+2\int_{A_{R, \ep}\cap\mC_{\la}}\lan\partial x_{0}, \partial\al\ran-\al e^{-2g-2x_{0}}|\partial Y_{0}|^{2} d\mu.$$
The integrals above converge to $0$ as $\la\rightarrow 0$ since $\al$ and $\frac{\partial x_{0}}{\partial n}$ are bounded and all the other integrands are uniformly integrable on $A_{R, \ep}\cap\mC_{\la}$.
\end{proof}

Let us return to the proof of Theorem \ref{main theorem}. Integrating inequality (\ref{convexity inequality1}) with respect to $t$ once, and using the fact that 
$\frac{d}{dt}\big{|}_{t=0}\mathcal{M}_{A_{R, \ep}}(x_{t}, Y_{t})=0$ we get,
$$\frac{d}{dt}\mathcal{M}_{A_{R, \epsilon}}(x_{t}, Y_{t}) \geq 2t\int_{A_{R, \epsilon}}|\partial d_{-1}\big((X, Y), (X_{0}, Y_{0})\big)|^{2}d\mu.$$
Integrating with respect to $t$ again, we get
$$ \mathcal{M}(x, Y)-\mathcal{M}(x_{0}, Y_{0}) \geq \int_{A_{R, \epsilon}}|\partial d_{-1}\big((X, Y), (X_{0}, Y_{0})\big)|^{2}d\mu.$$
Since the difference between $(x, Y)$ and $(x_{0}, Y_{0})$ is now restricted to a compact domain $B_{R}$, we can apply the scale invariant Sobolev inequality(see Theorem 1 on page 263 in \cite{E}) to get,
\begin{equation}\label{bound for L6 of d-1}
 \mathcal{M}(x, Y)-\mathcal{M}(x_{0}, Y_{0}) \geq \frac{1}{C}(\int_{A_{R, \epsilon}}|d_{-1}\big((X, Y), (X_{0}, Y_{0})\big)|^{6}d\mu)^{\frac{1}{3}}.
\end{equation}



In order to extend the above inequality to the general case $\alpha=x-x_{0}\in H^{1}(\mathbb{R}^{3})$ and $y=Y-Y_{0}\in H^{1}_{0}(\mathbb{R}^{3}\setminus\Gamma)$, we first use the compactly supported approximating sequence $\{(\al_{n}, y_{n})\}$ (\ref{approximation sequence}) into (\ref{bound for L6 of d-1}).
By basic hyperbolic geometry
\begin{displaymath}
\begin{split}
 d_{-1}\big( & (X, Y), (X_{n}, Y_{n})\big) =d_{-1}\big((X_{0}e^{\alpha}, Y_{0}+y), (X_{0}e^{\al_{n}}, Y_{0}+y_{n})\big)\\
  &\leq d_{-1}\big((X_{0}e^{\al}, Y_{0}+y), (X_{0}e^{\al}, Y_{0}+y_{n})\big)+d_{-1}\big((X_{0}e^{\al}, Y_{0}+y_{n}), (X_{0}e^{\al_{n}}, Y_{0}+y_{n})\big)\\
  &=e^{-\al}\frac{|y-y_{n}|}{X_{0}}+|\al-\al_{n}|\rightarrow 0, \quad \textrm{almost everywhere in $\R^{3}$},
\end{split}
\end{displaymath}
since $\al_{-}\in L^{\infty}$. Hence
$$|d_{-1}\big((X_{n}, Y_{n}), (X_{0}, Y_{0})\big)-d_{-1}\big((X, Y), (X_{0}, Y_{0})\big)|\rightarrow 0, \ \textrm{almost everywhere in $\R^{3}$}.$$
Using (\ref{approximation}) and Fatou's lemma to take the limit, we have proven (\ref{gap inequality}).
\end{proof}


\section{Extension to Chru\'sciel data}\label{chrusciel}

In this section we apply the convexity argument to the class of initial data defined in \cite{Chrus1}\cite{ChrusLiWe}. We first review the conditions on this data.

\subsection{Review of \cite{Chrus1}\cite{ChrusLiWe}}\label{review of chrusciel reduction}

Let us briefly review Chru\'sciel's reduction\cite{Chrus1}. Let $(M, g)$ be a \emph{3-dimensional simply connected asymptotically flat} manifold, say with two ends, such that each end $M_{ext}$ is diffeomorphic to $\R^{3}\setminus B(R)$. Assume that there are coordinates on 
$\R^{3}\setminus B(R)$ such that in these coordinates the metric $g$ satisfies,
\begin{equation}\label{decay of g at infinity}
g_{ij}-\de_{ij}=o_{k}(r^{-1/2}),\quad k\geq 5. 
\end{equation}
Assume $(M, g)$ is axisymmetric, i.e. there exists a killing vector field $\eta$ with complete periodic orbits, such that $\mL_{\eta}g=0$, then by Theorem 2.9 in \cite{Chrus1}, $M\simeq\R^{3}\setminus\{0\}$, where one end is at $\infty$ and the other at the origin $0$, and the metric $g$ can be written
\begin{equation}\label{metric}
g=e^{-2U+2\al}(d\rho^{2}+dz^{2})+\rho^{2}e^{-2U}(d\varphi+\rho B_{\rho}d\rho+A_{z}dz)^{2},
\end{equation}
where $(\rho, \varphi, z)$ are cylindrical coordinates of $\R^{3}$, and all functions are $\varphi$ independent. Furthermore, in these coordinates we have
\begin{equation}\label{axil-killing vector}
\eta=\partial_{\varphi},
\end{equation}
and
\begin{equation}\label{decay of U at infinity}
U=o_{k-3}(r^{-1/2}),\ \ r\rightarrow\infty,
\end{equation}
\begin{equation}
\al=o_{k-4}(r^{-1/2}),\ \ r\rightarrow\infty,
\end{equation}
\begin{equation}\label{rate of U at 0}
U=2\log r+o_{k-4}(r^{1/2}),\ \ r\rightarrow 0,
\end{equation}
\begin{equation}
\al=o_{k-4}(r^{1/2}),\ \ r\rightarrow 0.
\end{equation}

Now let $(M, g, h)$ be a \emph{simply connected, asymptotically flat, maximal, axisymmetric, vacuum initial data set} for the Einstein equations. We assume $(M, g)$ is as above, and we assume the asymptotic decay for $h$ on each end $M_{ext}$,
\begin{equation}\label{decay of k at infinity}
|h|_{g}=O_{k-1}(r^{-\la}),\ r\rightarrow \infty, \la>3/2.
\end{equation}

\begin{remark}
Note that our decay rate for $h$ is faster than $-3/2$, while in \cite{ChrusLiWe}, they require the decay rate to be faster than $-5/2$.
\end{remark}

Now the vacuum constraint equation for $(g, h)$ and the maximal condition $tr_{g}h=0$ imply $*_{g}(i_{\eta}h\wedge\eta)$ is closed\footnote{See Section 2 of \cite{D}.}, which is then exact since $\pi_{1}(M)=0$, so there exists a function $w$, such that,
\begin{equation}\label{momentum potential}
dw=*_{g}(i_{\eta}h\wedge\eta),
\end{equation}
where $*_{g}$ is the Hodge star operator for $g$. In our notation in Section \ref{introduction}
\begin{equation}\label{relation between (x, Y) and (U, w)}
U=-\frac{1}{2}x,\ \ w=\frac{1}{2}Y.
\end{equation}
It is obvious that $dw\equiv 0$ on the axis $\Ga=\{\rho=0,\ z\neq 0\}$ since $\eta\equiv 0$ there. We will normalize $w$ so that,
\begin{equation}\label{w on axis}
w|_{\mA_{i}}=w_{i},
\end{equation}
where $\mA_{1}=\{\rho=0, z<0\}$, $\mA_{2}=\{\rho=0, z>0\}$ are the two parts of the axis $\Ga$, and $w_{i}$ corresponds to the value of Extreme Kerr solution (\ref{Extreme Kerr}) on $\mA_{i}$.

Now by the decay (\ref{decay of g at infinity})(\ref{decay of k at infinity}) of $(g, h)$ and the definition of $dw$ (\ref{momentum potential}), we can derive the decay rate of $dw$ at infinity,
\begin{equation}\label{decay of Dw at infinity}
|Dw|_{\de}\leq C\rho^{2}r^{-\la},\ r\rightarrow\infty.
\end{equation}
By an inversion formula $x\rightarrow\frac{x}{|x|^{2}}$, which is done in (2.31)(2.32) in \cite{ChrusLiWe}, we can get the blow up rate of $dw$ near origin,
\begin{equation}\label{rate of Dw at 0}
|Dw|_{\de}\leq C^{\pr}\rho^{2}r^{\la-6},\ r\rightarrow 0.
\end{equation}
Using (\ref{momentum potential}) and (\ref{metric}) we have decay estimates of $dw$ near the axis away from $0$ and $\infty$,
\begin{equation}\label{decay of Dw at axis}
|Dw|_{\de}\leq C(\de)\rho^{2},\ \rho\rightarrow 0,\ \de\leq r\leq 1/\de,
\end{equation}
where $C(\de)$ is a constant depending on $\de$.

From (2.10) in \cite{ChrusLiWe}, we have a bound for the ADM mass $m$ of $(M, g, h)$ when $k\geq 6$,
\begin{equation}
m\geq \frac{1}{8\pi}\int_{\R^{3}}\big[|DU|^{2}+\frac{e^{4U}}{\rho^{4}}|Dw|^{2}\big]dx.
\end{equation}
Now we will apply the convexity argument to the functional
\begin{equation}\label{functional I}
\I(U, w):=\int_{\R^{3}}\big[|DU|^{2}+\frac{e^{4U}}{\rho^{4}}|Dw|^{2}\big]dx.
\end{equation}

\begin{theorem}\label{main theorem2}
For $k\geq 6$, $\I(U, w)$ is bounded from below by the corresponding value of the Extreme Kerr data(\ref{Extreme Kerr}), i.e. $\I_{0}=\I(U_{0}, w_{0})$, among all data $\{(U, w)\}$ satisfying (\ref{decay of U at infinity})(\ref{rate of U at 0})(\ref{decay of Dw at infinity}) (\ref{rate of Dw at 0})(\ref{decay of Dw at axis}) and (\ref{w on axis}), i.e.
\begin{equation}
\I(U, w)\geq \I(U_{0}, w_{0}).
\end{equation}
Moreover, we have the gap bound,
\begin{equation}\label{gap inequality2}
\I(U, w)-\I(U_{0}, w_{0})\geq C\{\int_{\R^{3}}d_{-1}^{6}\big((U, w),(U_{0}, w_{0})\big)dx\}^{1/3},
\end{equation}
where $d_{-1}\big((U, w), (U_{0}, w_{0})\big)$ is the distance between $(\rho^{2}e^{-2U}, 2w)$ and $(\rho^{2}e^{-2U_{9}}, 2w_{0})$ with respect to the hyperbolic metric $ds^{2}_{-1}$.
\end{theorem}

\begin{remark}\label{integrability of (U,w)}
Let us say a few words about the integrability of $\I(U, w)$ under conditions (\ref{decay of U at infinity})(\ref{rate of U at 0})(\ref{decay of Dw at infinity}) and (\ref{rate of Dw at 0}). In fact, near $\infty$, $|DU|^{2}=o(r^{-3})$ is integrable, and $\frac{e^{4U}}{\rho^{4}}|Dw|^{2}=O(r^{-2\la})$ is also integrable, when $\la>3/2$. Near the singularity $0$, $|DU|^{2}=O(r^{-2})$ is integrable, and $\frac{e^{4U}}{\rho^{4}}|Dw|^{2}=O(\frac{r^{8}}{\rho^{4}}\cdot \rho^{4}r^{2\la-12})=O(r^{2\la-4})$ which is integrable only when $\la>1/2$.

For the extreme Kerr solution $(U_{0}, w_{0})$, the blow up rate at the origin $0$ and decay rate at $\infty$ are\footnote{See Appendix A of \cite{ChrusLiWe}.}:
\begin{equation}\label{rate of U0 Dw0}
U_{0}=\log r+C,\ \ |Dw_{0}|_{\de}\leq C\frac{\rho^{2}}{r^{3}}:\ \ r\rightarrow 0.
\end{equation}
\begin{equation}\label{decay of Dw0}
|Dw_{0}|_{\de}\leq C\frac{\rho^{2}}{r^{3}}: \ \ r\rightarrow\infty.
\end{equation}
So the integrability of $\I(U_{0}, w_{0})$ follows as above.
\end{remark}


\subsection{Cut and paste argument}

Given data $(U, w)$ as in Theorem \ref{main theorem2}, the idea is that $\I(U, w)$ can be approximated by cutting and pasting $(U, w)$ to $(U_{0}, w_{0})$ near $\infty$, and then cutting and pasting $w$ to $w_{0}$ near $0$ and the axis $\Ga$. An idea of this type is used in \cite{ChrusLiWe}, but we take a different approximation here.
\begin{proposition}\label{approx prop}
Under conditions (\ref{decay of U at infinity})(\ref{rate of U at 0})(\ref{decay of Dw at infinity})(\ref{rate of Dw at 0})(\ref{decay of Dw at axis}) and (\ref{w on axis}) for $(U, w)$, for any small $c_{0}>0$ we can find $(U_{\de}, w_{\de, \ep})$ for small $\ep\ll\de\ll 1$, such that:
$$U_{\de}\equiv U,\ r<1/\de;\ w_{\de, \ep}\equiv w,\ \rho>\sqrt{\ep},\ 2\de<r<1/\de,$$
$$(U_{\de}, w_{\de, \ep})=(U_{0}, w_{0}), \ r>2/\de;\ w_{\de}\equiv w_{0},\ x\in B_{\de}\cup\mC_{\de, \ep},$$
where $\mC_{\de, \ep}$ is defined in (\ref{mC}), and
$$|\I(U, w)-\I(U_{\de}, w_{\de, \ep})|< c_{0}.$$
\end{proposition}

The proof is a combination of the following three lemmas. Let us define a family of smooth functions $\vp^{1}_{\de}\in C^{\infty}_{c}(\R^{3})$:
\begin{equation}\label{cutoff function1}
\vp^{1}_{\de}(r) \left\{ \begin{array}{ll}
=1 & \textrm{if $r\leq 1/\de$}\\
|D\vp^{1}_{\de}|\leq 2\de & \textrm{if $1/\de< r< 2/\de$}\\
=0 & \textrm{if $r\geq 2/\de$.}
\end{array} \right.
\end{equation}
Now define
$$U^{1}_{\de}=U_{0}+\vp^{1}_{\de}(U-U_{0}),\ \ w^{1}_{\de}=w_{0}+\vp^{1}_{\de}(U-U_{0}).$$
Then $(U^{1}_{\de}, w^{1}_{\de})\equiv (U_{0}, w_{0})$ outside $B_{2/\de}$.

\begin{lemma}\label{approx lemma1} We have
$\lim_{\de\rightarrow 0}\I(U^{1}_{\de}, w^{1}_{\de})=\I(U, w)$.
\end{lemma}
\begin{proof} We separate into three terms
$$\I(U^{1}_{\de}, w^{1}_{\de})=\underbrace{\int_{r\leq 1/\de}}_{I_{1}}+\underbrace{\int_{1/\de<r<2/\de}}_{I_{2}}+\underbrace{\int_{r\geq 2/\de}}_{I_{3}}\big[|DU^{1}_{\de}|^{2}+\frac{e^{4U^{1}_{\de}}}{\rho^{4}}|Dw^{1}_{\de}|^{2}\big]dx.$$
By the dominated convergence theorem(DCT\footnote{We will abbreviate DCT as dominant convergence theorem in the follow.}),
$$I_{1}=\int_{r\leq 1/\de}[|DU|^{2}+\frac{e^{4U}}{\rho^{4}}|Dw|^{2}]\rightarrow \I(U, w),$$
and
$$I_{3}=\int_{r\geq 2/\de}[|DU_{0}|^{2}+\frac{e^{4U_{0}}}{\rho^{4}}|Dw_{0}|^{2}\big]dx\rightarrow 0.$$
$$I_{2}=\underbrace{\int_{1/\de<r<2/\de}|DU^{1}_{\de}|^{2}dx}_{I_{21}}+\underbrace{\int_{1/\de<r<2/\de}\frac{e^{4U^{1}_{\de}}}{\rho^{4}}|Dw^{1}_{\de}|^{2}dx}_{I_{22}},$$
where
$$I_{21} \leq 2\int_{1/\de<r<2/\de}|DU|^{2}+|DU_{0}|^{2}+2\int_{1/\de<r<2/\de}\underbrace{(U-U_{0})^{2}}_{\sim o(r^{-1})}\underbrace{|D\vp^{1}_{\de}|^{2}}_{\leq 4\de^{2}}dx.$$
The first term converges to $0$ by DCT and remark \ref{integrability of (U,w)}, and the second term is asymptotic to $o(1)$ since $r\sim\de$ in this region, so it also converges to $0$. We also have
$$I_{22} \leq 4\int_{1/\de<r<2/\de}\frac{1}{\rho^{4}}(|Dw|^{2}+|Dw_{0}|^{2})+4\int_{1/\de<r<2/\de}\frac{1}{\rho^{4}}\underbrace{(w-w_{0})^{2}}_{\sim C\rho^{6}r^{-2\la}}\underbrace{|D\vp^{1}_{\de}|^{2}}_{\leq 4\de^{2}}dx.$$
This is because both $U$ and $U_{0}$ behave like $o(1)$ at infinity, so $e^{U^{1}_{\de}}$ is bounded by $2$ for $\de$ small enough. The first term converges to $0$ by DCT. The bound of 
$(w-w_{0})$ comes from the fact that $(w-w_{0})|_{\Ga}\equiv 0$ and an integration of (\ref{decay of Dw at infinity})(\ref{decay of Dw0}) along a line perpendicular to the axis $\Ga$. So the second term is asymptotic to $O(\de^{2\la-3})$ since $r\sim\de$, which converges to $0$ when $\la>3/2$. So we can get the limit by combining these results.
\end{proof}

Now we can first assume $U=U_{0}$ and $w=w_{0}$ outside a large ball $B_{R}$. Define a second family of smooth cutoff functions $\vp_{\de}\in C^{\infty}(\R^{3})$,
\begin{equation}\label{cutoff function2}
\vp_{\de}(r) \left\{ \begin{array}{ll}
=0 & \textrm{if $r\leq \de$}\\
|D\vp_{\de}|\leq 2/\de & \textrm{if $\de< r< 2\de$}\\
=1 & \textrm{if $r\geq 2\de$.}
\end{array} \right.
\end{equation}
We let
$$w_{\de}=w_{0}+\vp_{\de}(w-w_{0}).$$
Then $w_{\de}\equiv w_{0}$ inside the ball $B_{\de}$.

\begin{lemma}\label{approx lemma2} We have the result
$\lim_{\de\rightarrow 0}\I(U, w_{\de})=\I(U, w).$
\end{lemma}
\begin{proof} We consider three terms
$$\I(U, w_{\de})=\underbrace{\int_{r\leq \de}}_{I_{1}}+\underbrace{\int_{\de<r<2\de}}_{I_{2}}+\underbrace{\int_{r\geq 2\de}}_{I_{3}}|DU|^{2}+\frac{e^{4U}}{\rho^{4}}|Dw_{\de}|^{2}dx.$$
By DCT,
$$I_{3}=\int_{r\geq 2\de}|DU|^{2}+\frac{e^{4U}}{\rho^{4}}|Dw|^{2}\rightarrow \I(U, w).$$
On the other hand
$$I_{1}=\int_{r\leq\de}|DU|^{2}+\frac{1}{\rho^{4}}\underbrace{e^{4U}}_{\sim r^{8}}\underbrace{|Dw_{0}|^{2}}_{\sim\frac{\rho^{4}}{r^{6}}}dx.$$
The first term converges to $0$ by DCT. The second term, where we use (\ref{rate of U at 0})(\ref{rate of U0 Dw0}), is asymptotic to $\de^{5}$, hence converges to $0$. To handle $I_2$ we
estimate
\begin{displaymath}
\begin{split}
I_{2} &\leq\int_{\de<r<2\de}|DU|^{2}+2\frac{e^{4U}}{\rho^{4}}|Dw|^{2}+2\int_{\de<r<2\de}\frac{e^{4U}}{\rho^{4}}|Dw_{0}|^{2}\\
      &+2\int_{\de<r<2\de}\frac{1}{\rho^{4}}\underbrace{e^{4U}}_{\sim r^{8}}\underbrace{(w-w_{0})^{2}}_{\sim\rho^{6}r^{2\la-12}}\underbrace{|D\vp_{\de}|^{2}}_{\leq 4/\de^{2}}dx.
\end{split}
\end{displaymath}
The first term converges to $0$ by DCT. The second term converges to $0$ by the same argument as for $I_{1}$. The bound of $(w-w_{0})$ comes from $(w-w_{0})|_{\Ga}\equiv 0$ and an integration of (\ref{rate of Dw at 0})(\ref{rate of U0 Dw0}) along a line perpendicular to the axis 
$\Ga$. The last term is asymptotic to $O(\de^{2\la-1})$ since $r\sim\de$, which converges to $0$. Combining these together, we get the limit.
\end{proof}

\begin{remark}
The reason we can do this is because the blow-up rate($\rho^{4}r^{-6}$) of $|Dw_{0}|^{2}$ is smaller than that($\rho^{4}r^{2\la-12}$) of $|Dw|^{2}$ near the origin $0$, while the decay rate ($r^{8}$) of $e^{4U}$ is larger than that ($r^{4}$) of $e^{4U_{0}}$, so $|Dw_{0}|^{2}$ is also integrable with respect to $\frac{e^{4U}}{\rho^{4}}dx$ near the origin $0$.
\end{remark}

Besides assuming $(U, w)\equiv(U_{0}, w_{0})$ outside a large ball $B_{R}$, we can also assume $w\equiv w_{0}$ inside $B_{\de}$. Now define a third family of cutoff functions $\phi_{\ep}\in C^{\infty}(\R^{3})$,
\begin{equation}\label{cutoff function3}
\phi_{\ep}(\rho)= \left\{ \begin{array}{ll}
0 & \textrm{if $\rho\leq \ep$}\\
\frac{\ln(\rho/\ep)}{\ln(\sqrt{\ep}/\ep)} & \textrm{if $\ep<\rho<\sqrt{\ep}$}\\
1 & \textrm{if $\rho\geq\sqrt{\ep}$}
\end{array} \right.
\end{equation}
Define
$$w_{\ep}=w_{0}+\phi_{\ep}(w-w_{0}).$$
Define the sets
\begin{equation}\label{mC}
\mC_{\de, \ep}=\{\rho\leq\ep\}\cap\{\de\leq r\leq 2/\de\},
\end{equation}
\begin{equation}\label{mW}
\mW_{\de, \ep}=\{\ep\leq\rho\leq\sqrt{\ep}\}\cap\{\de\leq r\leq 2/\de\}.
\end{equation}
So we have $w_{\ep}\equiv w_{0}$ in $\mC_{\de, \ep}\cup B_{\de}$.

\begin{lemma}\label{approx lemma3} We have the limit 
$\lim_{\ep\rightarrow 0}\I(U, w_{\ep})\rightarrow \I(U, w).$
\end{lemma}
\begin{proof} We consider three terms
$$\I(U, w_{\ep})=\underbrace{\int_{\mC_{\de, \ep}}}_{I_{1}}+\underbrace{\int_{\mW_{\de, \ep}}}_{I_{2}}+\underbrace{\int_{\R^{3}\setminus\{\mC_{\de, \ep}\cup\mW_{\de, \ep}\}}}_{I_{3}}|DU|^{2}+\frac{e^{4U}}{\rho^{4}}|Dw_{\ep}|^{2}dx.$$
By $DCT$, $I_{3}\rightarrow\I(U, w)$.
$$I_{1}=\int_{\mC_{\de, \ep}}|DU|^{2}+\frac{e^{4U}}{\rho^{4}}\underbrace{|Dw_{0}|^{2}}_{\leq C\rho^{4}}dx.$$
The first term converges to $0$ by DCT, while the bound $|Dw_{0}|_{\de}$ come from (A.10) of \cite{ChrusLiWe}. The second term also converges to $0$ by DCT. To handle $I_2$ we estimate
\begin{displaymath}
\begin{split}
I_{2} &\leq \int_{\mW_{\de, \ep}}|DU|^{2}+2\frac{e^{4U}}{\rho^{4}}|Dw|^{2}+2\int_{\mW_{\de, \ep}}\frac{e^{4U}}{\rho^{4}}|Dw_{0}|^{2}\\
      &+2\int_{\mW_{\de, \ep}}\frac{e^{4U}}{\rho^{4}}\underbrace{(w-w_{0})^{2}}_{\leq C\rho^{4}}\underbrace{|D\phi_{\ep}|^{2}}_{\sim 1/(\rho\ln\ep)^{2}}dx.
\end{split}
\end{displaymath}
The first two terms converge to $0$ by DCT and the above argument as $\ep\rightarrow 0$. The bound of $(w-w_{0})$ is gotten by integrating $\partial_{\rho}(w-w_{0})$ along a line perpendicular to $\Ga$ with $(w-w_{0})|_{\Ga}\equiv 0$. So the last term is bounded by $C/|\ln\ep|$, which converges to $0$ as $\ep\rightarrow 0$. We have completed the proof.
\end{proof}


\subsection{Convexity and gap inequality}\label{convexity and gap inequality}

As in the first section, we denote
$$U=U_{0}+\al,\ \ w=w_{0}+y.$$
By Proposition \ref{approx prop}, we can first assume $(\al, y)$ is compactly supported in $B_{2/\de}$, and furthermore $y$ is compactly supported in $\Om_{\de, \ep}$, where
\begin{equation}\label{Omega de ep}
\Om_{\de, \ep}=\{\de<r<2/\de,\ \rho>\ep\}.
\end{equation}
Denote
\begin{equation}\label{mA de ep}
\mA_{\de, \ep}=B_{2/\de}\setminus\Om_{\de, \ep}.
\end{equation}
Now connect $(X=\rho^{2}e^{-2U}, 2w=2w_{0}+2y)$ to the Extreme Kerr data $(X_{0}=\rho^{2}e^{-2U_{0}}, Y_{0}=2w_{0})$(\ref{Extreme Kerr}) by a geodesic family $(X_{t}, 2w_{t})$ in $\mH^2$. Let $U_{t}=-\frac{1}{2}\ln X_{t}+\log\rho$ and $y_{t}=w_{t}-w_{0}$. Hence $w_{t}\equiv w_{0}$ in a neighborhood of $\mA_{\de, \ep}$, so $U_{t}=U_{0}+t\al$ in a neighborhood of $\mA_{\de, \ep}$ as discussed in Section 2. Then using the notation of Theorem \ref{main theorem2} we have the following result.
\begin{lemma}\label{convexity2} We have
\begin{equation}
\frac{d^{2}}{dt^{2}}\I(U_{t}, w_{t}) \geq \frac{1}{2}\int_{\R^{3}}|\nabla[d_{-1}\big((U, w), (U_{0}, w_{0})\big)]|^{2}dx.
\end{equation}
\end{lemma}
\begin{proof} We compute
\begin{displaymath}
\begin{split}
\frac{d^{2}}{dt^{2}}\I(U_{t}, w_{t}) &=\frac{d^{2}}{dt^{2}}\I_{B_{2/\de}}(U_{t}, w_{t})\\
                                     &=\underbrace{\frac{d^{2}}{dt^{2}}\I_{\Om_{\de, \ep}}(U_{t}, w_{t})}_{I_{1}}+\underbrace{\frac{d^{2}}{dt^{2}}\I_{\mA_{\de, \ep}}(U_{t}, w_{t})}_{I_{2}}.
\end{split}
\end{displaymath}
From equation (\ref{E and M}) we have $E_{\Om}(X, 2w)=4\I_{\Om}(U, w)+\int_{\partial\Om}\frac{\partial g}{\partial n}(g-4U)d\sigma$ on any compact domain $\Om$ of 
$\R^{3}\setminus \Ga$. The first term is calculated as in (\ref{first part of second variation of M}):
\begin{displaymath}
\begin{split}
I_{1} & =\frac{1}{4}\frac{d^{2}}{dt^{2}}E_{\Om_{\de, \ep}}(X_{t}, 2w_{t})+\frac{1}{4}\frac{d^{2}}{dt^{2}}\int_{\partial\Om_{\de, \ep}\cap\partial\mA_{\de, \ep}}\frac{\partial g}{\partial n}(g-4(U_{0}+t\al))d\sigma\\
      & \geq \frac{1}{2}\int_{\Om_{\de, \ep}}|\nabla[d_{-1}\big((U, w), (U_{0}, w_{0})\big)]|^{2}dx.
\end{split}
\end{displaymath}
Using the fact that $d_{-1}\big((U, w), (U_{0}, w_{0})\big)=2|\al|$ on $\mA_{\de, \ep}$, the second term is calculated as:
\begin{displaymath}
\begin{split}
I_{2} & =\frac{d^{2}}{dt^{2}}\int_{\mA_{\de, \ep}}|D(U_{0}+t\al)|^{2}+\frac{1}{\rho^{4}}e^{4(U_{0}+t\al)}|Dw_{0}|^{2}dx\\
      & =2\int_{\mA_{\de, \ep}}|D\al|^{2}+8\frac{1}{\rho^{4}}\al^{2}e^{4(U_{0}+t\al)}|Dw_{0}|^{2}dx\\
      & \geq \frac{1}{2}\int_{\mA_{\de, \ep}}|D[d_{-1}\big((U, w), (U_{0}, w_{0})\big)]|^{2}dx.
\end{split}
\end{displaymath}
Now let us check the validity for putting $\frac{d^{2}}{dt^{2}}$ into the $\int_{\mA_{\de, \ep}}$. We need to show the integrand after the second $``="$ is uniformly integrable for all $t\in[0, 1]$. The first term $\int_{\mA_{\de, \ep}}|D\al|^{2}dx$ is integrable since both $U, U_{0}\in H^{1}$. For the second term, let us separate $\mA_{\de, \ep}=B_{\de}\cup\mC_{\de, \ep}$. Then on $\mC_{\de, \ep}$, $\frac{1}{\rho^{4}}\underbrace{\al^{2}}_{bounded}\underbrace{e^{4(U_{0}+t\al)}}_{bounded}\underbrace{|Dw_{0}|^{2}}_{\sim\rho^{4}}$ is bounded, which is uniformly integrable. On $B_{\de}$, $\frac{1}{\rho^{4}}\underbrace{\al^{2}}_{\sim\log^{2}r}\underbrace{e^{4(U_{0}+t\al)}}_{\sim r^{4(1+t)}}\underbrace{|Dw_{0}|^{2}}_{\sim\rho^{4}r^{-6}}\leq  C(\log^{2}r)r^{-2}$ which is also uniformly integrable.

Combing these together, we get the convexity of the reduced energy $\I$ along geodesic paths.
\end{proof}

Let us check that the first variation at $(U_{0}, w_{0})$ is zero.
\begin{lemma}\label{vanishing of first variation2} We have
$\frac{d}{dt}\big{|}_{t=0}\I(U_{t}, w_{t})=0$.
\end{lemma}
\begin{proof}
By taking $\mu\ll\ep$ and $\la\ll\de$,
$$\frac{d}{dt}\Big{|}_{t=0}\I(U_{t}, w_{t})=\underbrace{\int_{\Om_{\la, \mu}}}_{I_{1}}+\underbrace{\int_{\mA_{\la, \mu}}}_{I_{2}}[2\lan DU_{0}, DU_{0}^{\pr}\ran+4U^{\pr}_{0}\frac{e^{4U_{0}}}{\rho^{4}}|Dw_{0}|^{2}+2\frac{e^{4U_{0}}}{\rho^{4}}\lan Dw_{0}, Dw_{0}^{\pr}\ran]dx.$$
Using integration by parts and the fact that $(U_{0}, w_{0})$ satisfies the Euler-Lagrange equation for $\I$ and that $(U_{0}^{\pr}, w_{0}^{\pr})=(\al, 0)$ in a neighborhood of $\mA_{\la, \mu}$, we have
$$I_{1}=\int_{\partial \mA_{\la, \mu}}2\frac{\partial}{\partial n}U_{0}\cdot\al.$$
Now separating $\mA_{\la, \mu}=B_{\la}\cup\mC_{\la, \mu}$,
\begin{displaymath}
\begin{split}
\frac{d}{dt}\Big{|}_{t=0}\I(U_{t}, w_{t}) &=\underbrace{\int_{\partial \mC_{\la, \mu}}2\frac{\partial}{\partial n}U_{0}\cdot\al d\sigma}_{I_{1}}+\underbrace{\int_{\mC_{\la, \mu}}2\lan DU_{0}, D\al\ran+4\al\frac{e^{4U_{0}}}{\rho^{4}}|Dw_{0}|^{2} dx}_{I_{2}}\\
                                    &+\underbrace{\int_{\partial B_{\la}}2\frac{\partial}{\partial n}U_{0}\cdot\al d\sigma}_{I_{3}}+\underbrace{\int_{B_{\la}}2\lan DU_{0}, D\al\ran+4\al\frac{e^{4U_{0}}}{\rho^{4}}|Dw_{0}|^{2} dx}_{I_{4}}.
\end{split}
\end{displaymath}
Since the equation above is always true for all $\mu\ll\ep$ and $\la\ll\de$, we can take a limit by first letting $\mu\rightarrow 0$, and then $\la\rightarrow 0$. For fixed $\la\ll\de$, the integrands in both $I_{1}$ and $I_{2}$ are bounded, so $I_{1}, I_{2}\rightarrow 0$ as $\mu\rightarrow 0$. Now $\underbrace{\frac{\partial}{\partial n}U_{0}}_{\sim r^{-1}}\cdot\underbrace{\al}_{\sim\log r}\underbrace{d\sigma}_{\sim r^{2}d\sigma_{0}}\sim r\log r d\sigma_{0}\rightarrow 0$\footnote{$d\sigma_{0}$ is the volume form on standard sphere.} as $\la\rightarrow 0$, hence $I_{3}\rightarrow 0$. $I_{4}$ converges to $0$ as $\la\rightarrow 0$, since both $DU_{0}$ and $D\al$ are $L^{2}$ integrable, and $\frac{1}{\rho^{4}}\underbrace{\al}_{\sim\log r}\underbrace{e^{4(U_{0})}}_{\sim r^{4}}\underbrace{|Dw_{0}|^{2}}_{\sim\rho^{4}r^{-6}}\sim (\log r)r^{-2}$ is also uniformly integrable. We have finished the proof of the lemma.
\end{proof}

\noindent\textbf{Proof of Theorem \ref{main theorem2}:}
Combining Lemma \ref{convexity2} and Lemma \ref{vanishing of first variation2}, integrating as in Section \ref{proof of main theorem1}, and using the Sobolev inequality(see \cite{E}), we can get:
\begin{displaymath}
\begin{split}
\I(U_{\de}, w_{\de, \ep})-\I(U_{0}, w_{0}) &\geq\frac{1}{4}\int_{\R^{3}}\big{|}D[d_{-1}\big((U_{\de}, w_{\de, \ep}), (U_{0}, w_{0})\big)]\big{|}^{2}dx\\
               &\geq C\{\int_{\R^{3}}d_{-1}^{6}\big((U_{\de}, w_{\de, \ep}), (U_{0}, w_{0})\big)dx\}^{1/3}.
\end{split}
\end{displaymath}
We will first take the limit as $\ep\rightarrow 0$, and then $\de\rightarrow 0$, then the left hand side will converge to $\I(U, w)-\I(U_{0}, w_{0})$ by Proposition \ref{approx prop}. Now we will show that the right hand side converges to $\{\int_{\R^{3}}d_{-1}^{6}\big((U, w), (U_{0}, w_{0})\big)dx\}^{1/3}$. By the triangle inequality, it suffices to show the following.
\begin{lemma}\label{L6 convergence of approx U and w} We have 
$\int_{\R^{3}}d_{-1}^{6}\big((U_{\de}, w_{\de, \ep}), (U, w)\big)dx\rightarrow 0$.
\end{lemma}
\begin{proof}
In fact,
\begin{displaymath}
\begin{split}
d_{-1}\big((U_{\de}, w_{\de, \ep}), (U, w)\big) &\leq d_{-1}\big((U_{\de}, w_{\de, \ep}), (U, w_{\de, \ep})\big)+d_{-1}\big((U, w_{\de, \ep}), (U, w)\big)\\
           &=2|U-U_{\de}|+2\frac{e^{2U}}{\rho^{2}}|w-w_{\de, \ep}|.
\end{split}
\end{displaymath}
Now we need to consider,
$$\int_{\R^{3}}(U-U_{\de})^{6}dx\sim\int_{\R^{3}\setminus B_{1/\de}}\underbrace{(U-U_{0})^{6}}_{\sim o(r^{-3})}dx,$$
which converges to $0$ as $\de\rightarrow 0$. Using asymptotic estimates as before,
\begin{displaymath}
\begin{split}
\int_{\R^{3}}\frac{e^{12U}}{\rho^{12}}(w-w_{\de, \ep})^{6}dx & \sim \int_{\R^{3}\setminus B_{1/\de}}\frac{1}{\rho^{12}}\underbrace{e^{12U}}_{\leq 2}\underbrace{(w-w_{0})^{6}}_{\sim \rho^{18}r^{-6\la}}dx +\int_{\mC_{\de, \ep}}\frac{1}{\rho^{12}}\underbrace{e^{12U}}_{\leq C}\underbrace{(w-w_{0})^{6}}_{\sim C\rho^{18}}dx\\
         &+\int_{B_{2\de}}\frac{1}{\rho^{12}}\underbrace{e^{12U}}_{\sim r^{24}}\underbrace{(w-w_{0})^{6}}_{\sim \rho^{18}r^{6\la-36}}dx.
\end{split}
\end{displaymath}
The second term is $\sim\ep^{8}$, and converges to $0$, when $\de$ fixed. The first term is $\sim\de^{6(\la-3/2)}$, which converges to $0$ for $\la>3/2$ when $\de\rightarrow 0$. The third term is 
$\sim\de^{6\la-3}$, and this converges to $0$ as $\de\rightarrow 0$.
\end{proof}


\section{Einstein Maxwell case}\label{e-m case}

Motivated by the work of P. Chru\'sciel and J. Costa \cite{ChrusCo} and G. Weinstein \cite{We}, we will extend the convexity and Sobolev bound to another renormalized harmonic energy functional corresponding to the axisymmetric vacuum Einstein/Maxwell equations.
For this purpose we consider the mapping $\ti{\Psi}=(u, v, \chi, \psi): \R^{3}\rightarrow\CH^{2}$, where $\CH^{2}=\{(u, v, \chi, \psi)\in\R^{4}\}$ is the complex hyperbolic plane with metric
$$ds_{\CH}^{2}=du^{2}+e^{4u}(dv+\chi d\psi-\psi d\chi)^{2}+e^{2u}(d\chi^{2}+d\psi^{2}).$$
The harmonic energy functional $E$ of $\ti{\Psi}:\Om\rightarrow \CH$ is
\begin{equation}\label{harmonic energy2}
E_{\Om}(\ti{\Psi})=\int_{\Om}|du|^{2}+e^{4u}|dv+\chi d\psi-\psi d\chi|^{2}+e^{2u}(|d\chi|^{2}+|d\psi|^{2})dx,
\end{equation}
where $\Om\subset\R^{3}$. Writing
\begin{equation}\label{u and U}
U=u+\log\rho,
\end{equation}
we can rewrite the above mapping as $\Psi=(U, v, \chi, \psi)$. We are interested in the following functional discussed in \cite{ChrusCo}\cite{Co},
\begin{equation}\label{reduced energy for EM}
\I_{\Om}(\Psi)=\int_{\Om}|DU|^{2}+\frac{e^{4U}}{\rho^{4}}|Dv+\chi D\psi-\psi D\chi|^{2}+\frac{e^{2U}}{\rho^{2}}(|D\chi|^{2}+|D\psi|^{2})dx,
\end{equation}
where $\Om\subset\R^{3}$, and we write $\I=\I_{\R^{3}}$. Now denote the one form $\om$ by
\begin{equation}\label{omega}
\om=Dv+\chi D\psi-\psi D\chi
\end{equation}
so that
\begin{equation}\label{reduced energy for EM2}
\I_{\Om}(\Psi)=\int_{\Om}|DU|^{2}+\frac{e^{4U}}{\rho^{4}}|\om|^{2}+\frac{e^{2U}}{\rho^{2}}(|D\chi|^{2}+|D\psi|^{2})dx.
\end{equation}
An result similar to (\ref{E and M}) can be derived by putting (\ref{u and U}) into (\ref{reduced energy for EM2}) and using integration by parts together with the fact that $\log\rho$ is harmonic on $\R^{3}\setminus\Ga$,
\begin{equation}\label{I and E}
\I_{\Om}(\Psi)=E_{\Om}(\ti{\Psi})+\int_{\partial\Om}\frac{\partial\log\rho}{\partial n}(2U+\log\rho)d\si,
\end{equation}
where $\Om$ is a compact region in $\R^{3}\setminus\Ga$, and $n$ is the unit outer normal of $\partial\Om$.

In fact, the extreme Kerr-Newman solution of the Einstein/Maxwell equations is a local critical point of $\I$\footnote{See \cite{We} for details.}. The extreme Kerr-Newman solution is determined by a map $\tilde{\Psi}_{0}=(u_{0}, v_{0}, \chi_{0}, \psi_{0})$, or equivalently $\Psi_{0}=(U_{0}, v_{0}, \chi_{0}, \psi_{0})$ with $U_{0}=u_{0}+\log\rho$, which is given (see \cite{ChrusN}, \cite{We}) as
\begin{equation}\label{Extreme Kerr-Newman}
\begin{split}
&u_{0}=-\frac{1}{2}\log\big[(\tilde{r}^{2}+a^{2}+\frac{a^{2}\sin^{2}\theta(2m\tilde{r}-q^{2})}{\Sigma})\sin^{2}\theta\big]\\
&v_{0}=ma\cos\theta(3-\cos^{2}\theta)-\frac{a(q^{2}\tilde{r}-ma^{2}\sin^{2}\theta)\cos\theta\sin^{2}\theta}{\Sigma}\\
&\chi_{0}=-\frac{qa\tilde{r}\sin^{2}\theta}{\Sigma}\\
&\psi_{0}=q\frac{(\tilde{r}^{2}+a^{2})\cos\theta}{\Sigma},
\end{split}
\end{equation}
where $m^{2}=a^{2}+q^{2}$, and
$$\tilde{r}=r+m,\ \Sigma=\tilde{r}^{2}+a^{2}\cos^{2}\theta.$$
Here $m$ is the ADM mass, $J=ma$ the angular-momentum, and $q$ the electric charge.

We are interested in the class of mappings $\Psi=(U, v, \chi, \psi)$ with finite reduced energy $\I(\Psi)<\infty$, which physically corresponds to axisymmetric initial data sets for the Einstein/Maxwell equations\footnote{See \cite{We} for initial data equation, and see \cite{ChrusCo}\cite{Co} for the relation between $\Psi$ and initial data.}. Here we will consider a class of maps which are variations from extreme Kerr-Newman map. Denote the difference $(\dU, \dv, \dch, \dps)$ by
\begin{equation}\label{dU dv dch dps}
\dU=U-U_{0},\ \dv=v-v_{0},\ \dch=\chi-\chi_{0},\ \dps=\psi-\psi_{0}.
\end{equation}
Motivated by the setting in \cite{D}, we consider the following restrictions on 
$(\dU, \dv, \dch, \dps)$,
\begin{equation}\label{conditions of dPsi in E/M case}
\begin{split}
& \dU\in H^{1}_{0}(\R^{3}),\ (\dU)_{+}\in L^{\infty}(\R^{3}),\\
& (\om-\om_{0})\in L^{2}_{0, \frac{e^{2U_{0}}}{\rho^{2}}}(\R^{3}),\\
& \dch, \dps\in H^{1}_{0, \frac{e^{U_{0}}}{\rho}}(\R^{3}),\ \frac{e^{U_{0}}}{\rho}\dch, \frac{e^{U_{0}}}{\rho}\dps\in L^{\infty}(\R^{3}),
\end{split}
\end{equation}
where $(\dU)_{+}$ denotes the positive part of $\d U$, and $H^{1}_{0, X}(\R^{3})$ is defined in (\ref{norm for y}).

\begin{remark}
This is a relatively restrictive requirement. We put it here in order to show a simple and direct proof compared to that in the next section.
\end{remark}

\begin{lemma}
Under condition (\ref{conditions of dPsi in E/M case}), $\I(\Psi)$ is finite.
\end{lemma}
\begin{proof}
Since $(\dU)_{+}\in L^{\infty}(\R^{3})$, we know that $\frac{e^{U}}{\rho}\leq C\frac{e^{U_{0}}}{\rho}$, so $H^{1}_{0, \frac{e^{U_{0}}}{\rho}}(\R^{3})\subset H^{1}_{0, \frac{e^{U}}{\rho}}(\R^{3})$ and $H^{1}_{0, \frac{e^{2U_{0}}}{\rho^{2}}}(\R^{3})\subset H^{1}_{0, \frac{e^{2U}}{\rho^{2}}}(\R^{3})$. The lemma now follows.
\end{proof}

\begin{lemma}
Under condition (\ref{conditions of dPsi in E/M case}), $\dv\in H^{1}_{0, X}(\R^{3})$, where $X$ is a smooth function defined on $\R^{3}\setminus\Ga$, with $X=\frac{e^{U_{0}}}{\rho}$ in a neighborhood of $\Ga$, and $X=\frac{e^{2U_{0}}}{\rho^{2}}$ elsewhere near $\infty$.
\end{lemma}
\begin{proof}
We compute
\begin{displaymath}
\begin{split}
\om& =(Dv+\chi D\psi-\psi D\chi)\\
   & =Dv_{0}+D\dv+(\chi_{0}+\dch)D(\psi_{0}+\dps)-(\psi_{0}+\dps)D(\chi_{0}+\dch)\\
   & =\om_{0}+(D\dv+\dch D\dps-\dps D\dch)\\
   & +(\dch D\psi_{0}-\dps D\chi_{0}+\chi_{0}D\dps-\psi_{0}D\dch).
\end{split}
\end{displaymath}
Therefore
\begin{displaymath}
\begin{split}
D\dv & =(\om-\om_{0})-(\dch D\dps-\dps D\dch)-(\dch D\psi_{0}-\dps D\chi_{0})\\
     & -\chi_{0}D\dps+\psi_{0}D\dch.
\end{split}
\end{displaymath}
In fact, from (\ref{conditions of dPsi in E/M case}) and the asymptotic behavior of $\Psi_{0}$ (See Appendix A in \cite{Co}), all terms except for $\psi_{0}D\dch$ lie in $L^{2}_{0, \frac{e^{2U_{0}}}{\rho^{2}}}(\R^{3})$, which are also in $L^{2}_{0, \frac{e^{U_{0}}}{\rho}}(\R^{3})$ near the axis $\Ga$, where $\frac{e^{U_{0}}}{\rho}\leq\frac{e^{2U_{0}}}{\rho^{2}}$. The last term $\psi_{0}D\dch$ lies in $L^{2}_{0, \frac{e^{U_{0}}}{\rho}}(\R^{3})$ since $\psi_{0}$ is bounded, so it also lies in $L^{2}_{0, \frac{e^{2U_{0}}}{\rho^{2}}}(\R^{3})$ as $\frac{e^{2U_{0}}}{\rho^{2}}\leq\frac{e^{U_{0}}}{\rho}$ elsewhere near $\infty$. Thus we have finished the proof.
\end{proof}

\begin{theorem}\label{main theorem3}
$\I(\Psi)$ has a global minimum at the Extreme Kerr-Newman $\Psi_{0}$, when $(\Psi-\Psi_{0})$ satisfies conditions (\ref{conditions of dPsi in E/M case}), i.e.
\begin{equation}
\I(\Psi)\geq\I(\Psi_{0}).
\end{equation}
Furthermore, we have the gap bound,
\begin{equation}\label{gap inequality3}
\I(\Psi)-\I(\Psi_{0})\geq C\{\int_{\R^{3}}d_{\CH}^{6}(\Psi, \Psi_{0})\}^{1/3}.
\end{equation}
\end{theorem}
\begin{proof}
The key point is that we can approximate $\dU$, $\dv$, $(\dch, \dps)$ by compactly supported smooth functions in $C^{\infty}_{c}(A_{R, \ep})$ and $C^{\infty}_{c}(\Om_{R, \ep})$ (see section \ref{singular case} for definition) under $H^{1}_{0}(\R^{3})$, $H^{1}_{0, X}(\R^{3})$, $H^{1}_{0, \frac{e^{U_{0}}}{\rho}}(\R^{3})$ norms respectively. Then the remainder of the proof is exactly the same as in the proof of Theorem \ref{main theorem} except that we use (\ref{I and E}) instead of (\ref{E and M}). We will address the details in next section.
\end{proof}


\section{Extension to Chru\'sciel-Costa data}\label{cc data}

Now we will extend the above result to a more general setting coming from physical asymptotic conditions described in \cite{ChrusCo}, \cite{Co}. In fact, we can handle weaker asymptotic conditions than \cite{ChrusCo}, \cite{Co}; for example, we need only assume $h, E, B=O_{k-1}(\frac{1}{r^{\la}})$ with $\la>\frac{3}{2}$\footnote{Compare to \cite{ChrusCo}\cite{Co}, where they assume $h=O(\frac{1}{r^{\be}})$ with $\be>\frac{5}{2}$, $E, B=O(\frac{1}{r^{1+\ga}})$ with $\ga>\frac{3}{4}$.}, where $h$, $E$ and $B$ are the second fundamental form, electric, and magnetic fields respectively.

In the notation described in the next section, we can state the main theorem which shows that 
$\Psi_{0}$ (extreme Kerr-Neuman) is the global minimum point of the reduced energy.
\begin{theorem}\label{main theorem4}
For $k\geq 6$, $\I(\Psi)$ is bounded from below by the corresponding value of the extreme Kerr-Newman map (\ref{Extreme Kerr-Newman}), i.e. for any map $\Psi=(U, v, \chi, \psi)$ satisfying (\ref{decay of U at infinity})(\ref{rate of U at 0})(\ref{decay of (om, chi, psi)})(\ref{rate of (om, chi, psi)}) (\ref{decay at axis}) and (\ref{axis condition}) we have
\begin{equation}
\I(\Psi)\geq\I(\Psi_{0}).
\end{equation}
Furthermore, we have the gap inequality,
\begin{equation}\label{gap inequality4}
\I(\Psi)-\I(\Psi_{0})\geq C\{\int_{\R^{3}}d_{\CH}^{6}(\Psi, \Psi_{0})dx\}^{1/3}.
\end{equation}
\end{theorem}


\subsection{Asymptotic behavior}\label{asymptotic}

We first describe the singular behavior of $\Psi$. From \cite{Chrus1}, we can assume $U$ satisfies (\ref{decay of U at infinity}) and (\ref{rate of U at 0}). From the asymptotic flatness conditions (see \cite{ChrusCo}, \cite{Co}) for corresponding initial data sets, we can assume the decay rate of $(\om, \chi, \psi)$  at $\infty$ is
\begin{equation}\label{decay of (om, chi, psi)}
|\om|=\rho^{2}O(r^{-\la});\ |D\chi|, |D\psi|=\rho O(r^{-\la}),\ r\rightarrow\infty,
\end{equation}
where we assume the decay rate of electric and magnetic fields is $O(r^{-\la})$\footnote{Compare with (2.3) in \cite{Co}.}. Now using an inversion near $0$,
\begin{equation}\label{rate of (om, chi, psi)}
|\om|=\rho^{2}O(r^{\la-6}); \ |D\chi|, |D\psi|=\rho O(r^{\la-4}),\ r\rightarrow 0.
\end{equation}
Near the axis $\Ga=\{\rho=0\}$, we can assume that,
\begin{equation}\label{decay at axis}
|\om|=O(\rho^{2});\ |D\chi|, |D\psi|=O(\rho),\ \rho\rightarrow 0,\ \de\leq r\leq 1/\de.
\end{equation}
Furthermore, we assume that the data corresponding to $\Psi$ has the same angular momentum and electric-magnetic charge as the extreme Kerr-Neuman data given by $\Psi_{0}$, i.e. they have the same value restricted to the axis $\Ga=\mA_{1}\cup\mA_{2}$\footnote{See discussion on page 4 in \cite{Co}. $\mA_{1}$ and $\mA_{2}$ are defined in section \ref{review of chrusciel reduction}.},
\begin{equation}\label{axis condition}
v|_{\Ga}=v_{0}|_{\Ga}=\left\{ \begin{array}{ll}
-2ma, & \textrm{on $\mA_{1}$}\\
2ma, & \textrm{on $\mA_{2}$}
\end{array} \right.,\
\chi|_{\Ga}=\chi_{0}|_{\Ga}=0,\
\psi|_{\Ga}=\psi_{0}|_{\Ga}=\left\{ \begin{array}{ll}
-q, & \textrm{on $\mA_{1}$}\\
q, & \textrm{on $\mA_{2}$}
\end{array} \right..
\end{equation}

Now let us derive more asymptotic conditions on the data. Using the boundary behavior (\ref{axis condition}) and integrating (\ref{decay of (om, chi, psi)}) along a line perpendicular to $\Ga$,
\begin{equation}\label{decay of chi, psi}
|\chi|=\rho^{2}O(r^{-\la}),\ |\psi|= const+\rho^{2}O(r^{-\la})=O(r^{-\la+2}),\ r\rightarrow\infty.
\end{equation}
Similarly integrating (\ref{rate of (om, chi, psi)}),
\begin{equation}\label{rate of chi, psi}
|\chi|=\rho^{2}O(r^{\la-4}),\ |\psi|=const+\rho^{2}O(r^{\la-4})=O(r^{\la-2}),\ r\rightarrow 0.
\end{equation}
Near the axis we can integrate (\ref{axis condition})
\begin{equation}\label{axis decay of chi, psi}
|\chi|=O(\rho^{2}),\ |\psi|=O(1),\ \rho\rightarrow 0,\ \de\leq r\leq 1/\de.
\end{equation}
Now combining with (\ref{omega})(\ref{decay of (om, chi, psi)})(\ref{rate of (om, chi, psi)}) and (\ref{decay of chi, psi})(\ref{rate of chi, psi})(\ref{axis decay of chi, psi}), we have
\begin{equation}\label{decay of Dv}
|Dv|\leq|\om|+|\chi D\psi-\psi D\chi|=\rho^{2}O(r^{-\la})+\rho O(r^{-2\la+2})=\rho O(r^{-\la+1}),\ r\rightarrow\infty.
\end{equation}
\begin{equation}\label{rate of Dv}
|Dv|\leq|\om|+|\chi D\psi-\psi D\chi|=\rho^{2}O(r^{\la-6})+\rho O(r^{2\la-6})=\rho O(r^{\la-5}),\ r\rightarrow 0.
\end{equation}
\begin{equation}\label{axis decay of Dv}
|Dv|\leq|\om|+|\chi D\psi-\psi D\chi|=O(\rho^{2})+O(\rho)=O(\rho),\ \rho\rightarrow 0,\ \de\leq r\leq 1/\de.
\end{equation}

\begin{remark}
Let us quickly review the integrability of $\I(\Psi)$. The $|DU|^{2}$ term is the same as in the vacuum case, and the term $\frac{e^{4U}}{\rho^{4}}|\om|^{2}$ is the same as $\frac{e^{4U}}{\rho^{4}}|dw|^{2}$ in Remark \ref{integrability of (U,w)}. Now for $(\chi, \psi)$, near $\infty$, $\frac{e^{2U}}{\rho^{2}}(|D\chi|^{2}+|D\psi|^{2})=O(r^{-2\la})$ is integrable for $\la>\frac{3}{2}$. Near $0$, $\frac{e^{2U}}{\rho^{2}}(|D\chi|^{2}+|D\psi|^{2})=O(r^{2\la-4})$ is also integrable.

Now let us also list the asymptotic behavior of $\Psi_{0}$
\begin{equation}\label{rate of (om0, chi0, psi0)}
|\om_{0}|=\rho^{2}O(r^{-3}),\ |D\chi_{0}|=\rho O(r^{-3}), |D\psi_{0}|=\rho O(r^{-2}),\ \chi_{0}=\rho^{2}O(r^{-3}),\ \psi_{0}=O(1),\ r\rightarrow \infty.
\end{equation}
\begin{equation}\label{decay of (om0, chi0, psi0)}
|\om_{0}|=\rho^{2}O(r^{-3}),\ |D\chi_{0}|, |D\psi_{0}|=\rho O(r^{-2}),\ \chi_{0}=\rho^{2}O(r^{-2}),\ \psi_{0}=O(1),\ r\rightarrow 0.
\end{equation}
Here the behavior of $\om$ is gotten by direct calculations based on (\ref{Extreme Kerr-Newman}), and other calculations can be found in Appendix A in \cite{Co}.
\end{remark}


\subsection{Cut and paste argument}

Given $\Psi=(U, v, \chi, \psi)$ as in Theorem \ref{main theorem4}, we approximate $\I(\Psi)$ again by cutting and pasting $\Psi$ to $\Psi_{0}$ near $\infty$, and then cutting and pasting $(v, \chi, \psi)$ to $(v_{0}, \chi_{0}, \psi_{0})$ near $0$ and axis $\Ga$.
\begin{proposition}\label{approx prop2}
Under conditions (\ref{decay of U at infinity})(\ref{rate of U at 0})(\ref{decay of (om, chi, psi)})(\ref{rate of (om, chi, psi)})(\ref{decay at axis}) and (\ref{axis condition}) for $\Psi=(U, v, \chi, \psi)$, for any small $c_{0}>0$, we can find $\Psi_{\de, \ep}=(U_{\de}, v_{\de, \ep}, \chi_{\de, \ep}, \psi_{\de, \ep})$ for small $\ep\ll\de\ll 1$, such that:
$$U_{\de}\equiv U,\ r<1/\de;\ (v_{\de, \ep}, \chi_{\de, \ep}, \psi_{\de, \ep})\equiv (v, \chi,\ \psi),\ \rho>\sqrt{\ep},\ 2\de<r<1/\de,$$
$$(U_{\de}, v_{\de, \ep}, \chi_{\de, \ep}, \psi_{\de, \ep})=(U_{0}, v_{0}, \chi_{0}, \psi_{0}), \ r>2/\de,$$
$$(v_{\de, \ep}, \chi_{\de, \ep}, \psi_{\de, \ep})\equiv (v_{0}, \chi_{0}, \psi_{0}),\ x\in B_{\de}\cup\mC_{\de, \ep},$$
where $\mC_{\de, \ep}$ is defined in (\ref{mC}), and
$$|\I(\Psi)-\I(\Psi_{\de, \ep})|<c_{0}.$$
\end{proposition}

As in the vacuum case, we can achieve this approximation is three steps. Now we will sketch the proof. First define
$$\Psi^{1}_{\de}=\Psi_{0}+\vp^{1}_{\de}(\Psi-\Psi_{0}),$$
where $\vp^{1}_{\de}$ is defined in (\ref{cutoff function1}). Then $\Psi^{1}_{\de}=\Psi_{0}$ outside $B_{2/\de}$.

\begin{lemma}
$\lim_{\de\rightarrow 0}\I(\Psi^{1}_{\de})=\I(\Psi)$.
\end{lemma}
\begin{proof}
By comparing to the proof of lemma \ref{approx lemma1}, the only difference from that case is to show
$$\int_{1/\de<r<2/\de}\frac{e^{4U^{1}_{\de}}}{\rho^{4}}|\om^{1}_{\de}|^{2}dx\rightarrow 0,$$
where (by (2.16) in \cite{Co})
\begin{displaymath}
\begin{split}
\om^{1}_{\de} &=\vp^{1}_{\de}\om+(1-\vp^{1}_{\de})\om_{0}+\underbrace{D\vp^{1}_{\de}(v-v_{0})}_{\sim \de\rho^{2} r^{-\la+1}} +\underbrace{D\vp^{1}_{\de}(\chi_{0}\psi-\psi_{0}\chi)}_{\sim \de\rho^{2}r^{-\la}}\\
              &+\vp^{1}_{\de}(1-\vp^{1}_{\de})\underbrace{\{(\psi-\psi_{0})D(\chi-\chi_{0})-(\chi-\chi_{0})D(\psi-\psi_{0})\}}_{\sim \rho^{2}r^{-2\la+1}}.
\end{split}
\end{displaymath}
The asymptotic behavior comes from (\ref{decay of Dv})(\ref{decay of chi, psi})(\ref{axis condition})(\ref{decay of (om, chi, psi)}) and those of Extreme-Kerr coming from Appendix A in \cite{Co}. Convergence follows from the asymptotics.
\end{proof}

Now we can assume $\Psi=\Psi_{0}$ outside $B_{2/\de}$. Define
$$(v_{\de}, \chi_{\de}, \psi_{\de})=(v_{0}, \chi_{0}, \psi_{0})+\vp_{\de}(v-v_{0}, \chi-\chi_{0}, \psi-\psi_{0}),$$
where $\vp_{\de}$ is defined in (\ref{cutoff function2}). Then $(v_{\de}, \chi_{\de}, \psi_{\de})=(v_{0}, \chi_{0}, \psi_{0})$ in $B_{\de}$. Let $\Psi_{\de}=(U, v_{\de}, \chi_{\de}, \psi_{\de})$.

\begin{lemma} We have
$\lim_{\de\rightarrow 0}\I(\Psi_{\de})=\I(\Psi)$.
\end{lemma}
\begin{proof}
By comparing to the proof of lemma \ref{approx lemma2}, the different term we need to handle is,
$$\int_{\de<r<2\de}\frac{e^{4U}}{\rho^{4}}|\om_{\de}|^{2}dx\rightarrow 0,$$
while
\begin{displaymath}
\begin{split}
\om_{\de} &=\vp_{\de}\om+(1-\vp_{\de})\om_{0}+\underbrace{D\vp_{\de}(v-v_{0})}_{\sim (1/\de)\rho^{2} r^{\la-5}} +\underbrace{D\vp_{\de}(\chi_{0}\psi-\psi_{0}\chi)}_{\sim (1/\de)\rho^{2}r^{\la-4}}\\
              &+\vp_{\de}(1-\vp_{\de})\underbrace{\{(\psi-\psi_{0})D(\chi-\chi_{0})-(\chi-\chi_{0})D(\psi-\psi_{0})\}}_{\sim \rho^{2}r^{2\la-7}},
\end{split}
\end{displaymath}
where the asymptotics come from (\ref{rate of Dv})(\ref{rate of chi, psi})(\ref{rate of (om, chi, psi)})(\ref{axis condition}). Convergence follows from the asymptotics and the fact that $\frac{e^{4U}}{\rho^{4}}\sim\frac{r^{8}}{\rho^{4}}$.
\end{proof}
\begin{remark}
The reason we can improve to $\la>\frac{3}{2}$ (weaker than \cite{ChrusCo}, \cite{Co}) is that $e^{4U}\sim r^{8}$ by (\ref{rate of U at 0}) is faster than $e^{4U_{0}}\sim r^{4}$ by (\ref{rate of U0 Dw0}), while we did not cut $U$ off near $0$.
\end{remark}

Now we can assume furthermore that $(v, \chi, \psi)=(v_{0}, \chi_{0}, \psi_{0})$ in $B_{\de}$. Define
$$(v_{\ep}, \chi_{\ep}, \psi_{\ep})=(v_{0}, \chi_{0}, \psi_{0})+\phi_{\ep}(v-v_{0}, \chi-\chi_{0}, \psi-\psi_{0}),$$
with $\phi_{\ep}$ defined in (\ref{cutoff function3}). Now $(v_{\ep}, \chi_{\ep}, \psi_{\ep})=(v_{0}, \chi_{0}, \psi_{0})$ in $\mC_{\de, \ep}\cup B_{\de}$. Denote $\Psi_{\ep}=(U, v_{\ep}, \chi_{\ep}, \psi_{\ep})$.

\begin{lemma} We have 
$\lim_{\ep\rightarrow 0}\I(\Psi_{\ep})=\I(\Psi)$.
\end{lemma}
\begin{proof}
By comparing to the proof of lemma \ref{approx lemma3}, the additional term we need to handle is,
$$\int_{\mW_{\de, \ep}}\frac{e^{4U}}{\rho^{4}}|\om_{\ep}|^{2}dx\rightarrow 0,$$
while
\begin{displaymath}
\begin{split}
\om_{\ep} &=\phi_{\ep}\om+(1-\phi_{\ep})\om_{0}+\underbrace{D\phi_{\ep}(v-v_{0})}_{\sim (1/(\rho\ln\ep))\rho^{2}} +\underbrace{D\phi_{\ep}(\chi_{0}\psi-\psi_{0}\chi)}_{\sim (1/(\rho\ln\ep))\rho^{2}}\\
          &+\phi_{\ep}(1-\phi_{\ep})\underbrace{\{(\psi-\psi_{0})D(\chi-\chi_{0})-(\chi-\chi_{0})D(\psi-\psi_{0})\}}_{\sim \rho^{3}},
\end{split}
\end{displaymath}
where the asymptotics come from (\ref{axis decay of Dv})(\ref{axis decay of chi, psi})(\ref{decay at axis})(\ref{axis condition}). Convergence follows from these asymptotics.
\end{proof}
Combining the above three lemmas, we have proven Proposition \ref{approx prop2}.


\subsection{Convexity and gap inequality}

The proof of Theorem \ref{main theorem4} is very similar to that in Section \ref{convexity and gap inequality}. We will point out the main differences here. By Proposition \ref{approx prop2}, we can first take $(\dU, \dv, \dch, \dps)$ in (\ref{dU dv dch dps}) to satisfy: $(1)$ $\dU$ is compactly supported in $B_{2/\de}$; $(2)$ $(\dv, \dch, \dps)$ are compactly supported in $\Om_{\de, \ep}$, which is defined in (\ref{Omega de ep}).

Now we can connect $\ti{\Psi}=(u=U-\log\rho, v, \chi, \psi)$ to $\ti{\Psi}_{0}=(u_{0}=U_{0}-\log\rho, v_{0}, \chi_{0}, \psi_{0})$ by a geodesic family $\ti{\Psi}_{t}=(u_{t}, v_{t}, \chi_{t}, \psi_{t})$ on $(\CH^{2}, ds^{2}_{\CH})$. Denote $U_{t}=u_{t}+\log\rho$. We know that $\Psi_{t}\equiv\Psi_{0}$ outside $B_{2/\de}$. Then $(v_{t}, \chi_{t}, \psi_{t})\equiv(v_{0}, \chi_{0}, \psi_{0})$ in a neighborhood of $\mA_{\de, \ep}$ (defined in (\ref{mA de ep})). So $U_{t}=U_{0}+t\dU$ in a neighborhood of $\mA_{\de, \ep}$ as in Section \ref{convexity}. As in Lemma \ref{convexity2}, we have
\begin{lemma}\label{convexity3} The following inequality holds
\begin{equation}
\frac{d^{2}}{dt^{2}}\I(\Psi_{t})\geq 2\int_{\R^{3}}|D\big(d_{\CH}(\Psi, \Psi_{0})\big)|^{2}dx.
\end{equation}
\end{lemma}
\begin{proof}
\begin{equation}
\begin{split}
\frac{d^{2}}{dt^{2}}\I(\Psi_{t}) &=\frac{d^{2}}{dt^{2}}\I_{B_{2/\de}}(\Psi_{t})\\
                                 &=\underbrace{\frac{d^{2}}{dt^{2}}\I_{\Om_{\de, \ep}}(\Psi_{t})}_{I_{1}} +\underbrace{\frac{d^{2}}{dt^{2}}\I_{\mA_{\de, \ep}}(\Psi_{t})}_{I_{2}}.
\end{split}
\end{equation}
Using formula (\ref{I and E}), the fact that $(\CH^{2}, ds_{\CH}^{2})$ is negatively curved and (\ref{refined second variation for E}), the first part is calculated as:
\begin{equation}
\begin{split}
I_{1} & =\frac{d^{2}}{dt^{2}}E_{\Om_{\de, \ep}}(\ti{\Psi}_{t})+\frac{d^{2}}{dt^{2}}\int_{\partial\Om_{\de, \ep}\cap\partial\mA_{\de, \ep}}\frac{\partial\log\rho}{\partial n}(2(U_{0}+t\dU)+\log\rho)d\si\\
      & \geq 2\int_{\Om_{\de, \ep}}|D\big(d_{\CH}(\Psi, \Psi_{0})\big)|^{2}dx.
\end{split}
\end{equation}
Since $d_{\CH}(\Psi, \Psi_{0})=|\dU|$ on $\mA_{\de, \ep}$, the second part is calculated as:
\begin{equation}
\begin{split}
I_{2} & =\frac{d^{2}}{dt^{2}}\int_{\mA_{\de, \ep}}|D(U_{0}+t\dU)|^{2}+\frac{e^{4(U_{0}+t\dU)}}{\rho^{4}}|\om_{0}|^{2}+\frac{e^{2(U_{0}+t\dU)}}{\rho^{2}}(|D\chi_{0}|^{2}+|D\psi_{0}|^{2})dx\\
      & =2\int_{\mA_{\de, \ep}}|D\dU|^{2}+8(\dU)^{2}\frac{e^{4(U_{0}+t\dU)}}{\rho^{4}}|\om_{0}|^{2}+2(\dU)^{2}\frac{e^{2(U_{0}+t\dU)}}{\rho^{2}}(|D\chi_{0}|^{2}+|D\psi_{0}|^{2})dx\\
      & \geq 2\int_{\mA_{\de, \ep}}|D\big(d_{\CH}(\Psi, \Psi_{0})\big)|^{2}dx.
\end{split}
\end{equation}
Now the reason that we can take $\frac{d^{2}}{dt^{2}}$ into the integral in the second $``="$ follows from the same idea as in the proof of \ref{convexity2}, making use of (\ref{decay of (om0, chi0, psi0)}). For example, $\underbrace{(\dU)^{2}\frac{e^{4(U_{0}+t\dU)}}{\rho^{4}}|\om_{0}|^{2}}_{\sim(\log r)^{2}\frac{r^{4(1+t)}}{\rho^{4}}\rho^{4} r^{-6}}\sim(\log r)^{2}r^{-2}$ is uniformly integrable near $0$. Other terms follow similarly. We have proven the lemma.
\end{proof}

Using the idea in Lemma \ref{vanishing of first variation2}, while using the fact that $\Psi_{0}$ satisfies the Euler-Lagrange equation for $\I$, we can easily get the following result. We omit the proof here since it is almost the same as Lemma \ref{vanishing of first variation2}.
\begin{lemma}\label{vanishing of first variation3} At $t=0$ we have
$\frac{d}{dt}\big{|}_{t=0}\I(\Psi_{t})=0$.
\end{lemma}

\noindent\textbf{Proof of Theorem \ref{main theorem4}:}
The proof follows exactly the same idea as the proof of Theorem \ref{main theorem2} by using Proposition \ref{approx prop2}, Lemma \ref{convexity3}, and Lemma \ref{vanishing of first variation3}. We leave details to the reader.


\parindent 0ex
Department of Mathematics, Stanford University, building 380\\
Stanford, California 94305\\
E-mail: schoen$@$math.stanford.edu, xzhou08$@$math.stanford.edu.

\end{document}